\documentclass[reqno]{amsart}
\usepackage{amssymb,amsmath,enumerate,mathrsfs}

\numberwithin{equation}{section}

\newtheorem{theorem}[equation]{Theorem}
\newtheorem{proposition}[equation]{Proposition}
\newtheorem{lemma}[equation]{Lemma}

\theoremstyle{definition}
\newtheorem{definition}[equation]{Definition}

\def\C{\mathbb C}
\def\N{\mathbb N}
\def\R{\mathbb R}
\def\L{\mathscr L}
\def\S{\mathscr S}
\def\opm{\textup{op}_{\textit{M}}}

\DeclareMathOperator{\ran}{ran}
\DeclareMathOperator{\sym}{ \sigma\!\!\!\sigma}

\begin{document}
\title[Shubin calculus on asymptotically conic manifolds]{A Shubin pseudodifferential calculus on asymptotically conic manifolds}

\author{Thomas Krainer}
\address{Penn State Altoona\\ 3000 Ivyside Park \\ Altoona, PA 16601-3760}
\email{krainer@psu.edu}
\thanks{The author gratefully acknowledges the financial support and hospitality of the \emph{Max Planck Institute for Mathematics} in Bonn during May 2024 where much of this work was completed.}

\begin{abstract}
We present a global pseudodifferential calculus on asymptotically conic manifolds that generalizes (anisotropic versions of) Shubin's classical global pseudodifferential calculus on Euclidean space to this class of noncompact manifolds. Fully elliptic operators are shown to be Fredholm in an associated scale of Sobolev spaces, and to have parametrices in the calculus.
\end{abstract}


\maketitle


\section{Introduction}

\noindent
In \cite{ShubinCalc} Shubin introduced a global pseudodifferential calculus on Euclidean space $\R^d$ that is modeled on the quantum harmonic oscillator
$$
\Delta + |y|^2 : \S(\R^d) \to \S(\R^d),
$$
where $\Delta = \sum_{j=1}^d D_{y_j}^2$ is the positive Laplacian in $\R^d$. Its symbol is $p(y,\eta) = |y|^2 + |\eta|^2$, and the main idea in Shubin's calculus is to remove the distinction between variables and covariables on the level of symbols and consider symbol estimates of the form
\begin{equation}\label{globsymb}
|\partial_{(y,\eta)}^{\alpha}p(y,\eta)| \lesssim (1 + |y| + |\eta|)^{\mu-|\alpha|},
\end{equation}
where for the harmonic oscillator $\mu = 2$. Smooth functions $p_{(\mu)}(y,\eta) \in C^{\infty}(\R^{2d}\setminus \{(0,0)\})$ that are jointly homogeneous in $(y,\eta)$ in the sense that
$$
p(\varrho y,\varrho \eta) = \varrho^{\mu}p(y,\eta), \; \varrho > 0,\, (y,\eta) \neq (0,0),
$$
give rise to symbols of order $\mu \in \R$ after multiplication by an excision function $\chi \in C^{\infty}(\R^{2d})$ of the origin, and classical symbols are those with an asymptotic expansion
$$
p(y,\eta) \sim \sum\limits_{j=0}^{\infty}\chi(y,\eta) p_{(\mu-j)}(y,\eta),
$$
where $p_{(\mu-j)}$ is jointly homogeneous of degree $\mu-j$. The classical symbol $p(y,\eta)$ is fully elliptic if $p_{(\mu)}(y,\eta)$ is invertible for $(y,\eta) \in \R^{2d}\setminus\{(0,0)\}$.

The Shubin calculus has many desirable properties of a pseudodifferential calculus \cite{NicolaRodino,ShubinBook}: It is closed with respect to taking compositions and formal adjoints; it comes with a scale of Sobolev spaces, and fully elliptic operators are Fredholm in that Sobolev space scale; finally, fully elliptic operators have parametrices in the calculus modulo regularizing remainders that are integral operators with kernels in $\S(\R^d\times\R^d)$.
On the other hand, the Shubin calculus does not have good perturbation and invariance properties. For example, from the symbol estimates \eqref{globsymb} it is easy to see that the only differential operators in the Shubin calculus are those with polynomial coefficients, which reveals that the calculus is not well-behaved with respect to changes of coordinates.

In this paper we present a pseudodifferential calculus that generalizes and extends the Shubin calculus from $\R^d$ to asymptotically conic manifolds. When specialized to $\R^d$ our calculus is broader than the Shubin calculus, which is necessary to overcome the restrictions with respect to coordinate changes; in particular, our calculus contains many more differential operators on $\R^d$ than Shubin's calculus, and it enjoys improved perturbation properties. Our approach is quite different from the quadratic scattering calculus \cite{Wunsch} which is also modeled on generalizations of the harmonic oscillator to such manifolds.

More precisely, we define here a family of pseudodifferential calculi on asymptotically conic manifolds whose members are distinguished by the anisotropic treatment of covariables and coefficient growth at infinity. Fully elliptic model operators on $\R^d$ for this family are generalized anharmonic oscillators of the form
$$
\Delta^m + |y|^{2n} : \S(\R^d) \to \S(\R^d)
$$
for $m,n \in \N$ as considered in the pseudodifferential calculus in \cite{ChatzakouDelgadoRuzhansky,ChatzakouDelgadoRuzhansky2}, see also \cite{NicolaRodino}.
To define our pseudodifferential calculus we use Mellin pseudodifferential operators with respect to the radial variable near infinity. In order to deal with tempered distributions, the symbols of these Mellin operators are required to be holomorphic in the radial covariable. Mellin pseudodifferential operators with holomorphic symbols are an important ingredient in Schulze's pseudodifferential operator theory on singular manifolds \cite{SchuNH}, and they are utilized in that theory for quantizing symbols near the singularities rather than at infinity as in our situation. We are nonetheless able to take advantage in this paper of some techniques that were originally developed for Schulze's theory \cite{GilSchulzeSeiler0,GilSchulzeSeiler,SchuNH}. We also previously developed a Mellin pseudodifferential calculus for the functional analysis of ordinary differential operators of Fuchs type with unbounded operator coefficients in \cite{KrainerIndicial} that relates to the present work as well.

In this paper we focus on the construction of the calculus and the investigation of ellipticity, parametrices, and the Fredholm property. In forthcoming work we plan to analyze resolvent families, trace asymptotics, and $\zeta$-functions for fully elliptic operators in our calculus with the goal to extend results about the spectral theory of such operators \cite{HelfferAst,NicolaRodino,Parmeggiani} from $\R^d$ to asymptotically conic manifolds.


\section{Asymptotically conic manifolds and operators}\label{ManEnds}

\noindent
Let $M_0$ be a smooth compact manifold of dimension $d$ with boundary $Z$, and let $U$ be a collar neighborhood of $Z$ in $M_0$ so that $U \cong (\tfrac{1}{2},1]_x \times Z$ with the boundary $Z$ at $x=1$. We then get a noncompact $C^{\infty}$-manifold $M$ of dimension $d$ by gluing $M_0$ to $[1,\infty)\times Z$ along the common boundary $Z$:
$$
M = M_0 \sqcup_Z ([1,\infty) \times Z).
$$
Pick a $C^{\infty}$-function $x: M \to \R$ such that $0 < x \leq \frac{1}{2}$ on $M_0 \setminus U$, and such that $x$ coincides with the projection to the first factor in $U \cong (\tfrac{1}{2},1] \times Z$, and on the end $[1,\infty)\times Z$. A (model) asymptotically conic metric on $M$ is a Riemannian metric $g$ that for large $x \gg 1$ on the end can be written in the form
\begin{equation}\label{ModelMetric}
g = dx^2 + x^2g_Z
\end{equation}
for some Riemannian metric $g_Z$ on $Z$. If $(Z,g_Z) = ({\mathbb S}^{d-1},g_{{\mathbb S}^{d-1}})$ is the unit sphere in $\R^d$ with its standard metric, then the metric $g$ for large $x \gg 1$ corresponds to the Euclidean metric in $\R^d$ written in polar coordinates (the manifold $M$ is then asymptotically Euclidean). The positive Laplace-Beltrami operator on $M$ with respect to $g$ for large $x \gg 1$ is given by
$$
\Delta_g = D_x^2 - i\frac{d-1}{x}D_x + \frac{1}{x^2}\Delta_{g_Z} = x^{-2}\bigl[(xD_x)^2 - i(d-2)(xD_x) + \Delta_{g_Z} \bigr],
$$
and for the $L^2$-space with respect to $g$ on $M$ we find that
$$
L^2([R,\infty)\times Z,g) = x^{-\frac{d}{2}}L^2([R,\infty)\times Z,\tfrac{dx}{x}dz) = x^{-\frac{d}{2}}L^2_b([R,\infty)\times Z)
$$
for $R \gg 1$ large enough, where $dz$ is the Riemannian density with respect to the metric $g_Z$. As the methods in our paper are based on the Mellin transform on the asymptotically conic end with respect to the variable $x$, and the space $L^2_b$ based on Haar measure on $\R_+$ is the preferred space for working with this transform, we adopt the point of view that the geometric $L^2$-space is a weighted $L^2_b$ space, thus declaring $L^2_b$ to be the base space for all considerations in this paper.

We next give some examples for the type of differential operators we aim to analyze on $M$ and discuss some features of our approach:

\subsubsection*{Harmonic oscillator}

Let $V \in C^{\infty}(M)$ be a potential such that $V = x^2$ for large $x \gg 1$. The resulting operator $\Delta_g + V$ is an analogue of the harmonic oscillator on $M$. For large $x$ we find
\begin{equation}\label{HarmonicOscillator}
\Delta_g + V = x^{-2}A = x^{-2}\bigl[(xD_x)^2 - i(d-2)(xD_x) + \Delta_{g_Z} + x^4 \bigr].
\end{equation}
Our goal is to prove Fredholmness results for such operators on $M$ in appropriate scales of weighted Sobolev spaces. We are going to construct a pseudodifferential calculus on $M$ in which such operators are fully elliptic, and the calculus contains the (Fredholm) inverses and parametrices.

To illustrate the associated notion of ellipticity in the calculus consider the operator $A$ from \eqref{HarmonicOscillator}. It has the usual homogeneous principal symbol given by
$$
\sym_{\psi}(A) = (x \xi)^2 + |\zeta|_{g_Z}^2
$$
for large $x \gg 1$, which is invertible for nonzero covectors. We are rescaling $\sigma = x\xi$, which is familiar from the $b$-calculus (but here for large values of $x$) and is based on the Mellin correspondence $xD_x \cong \sigma$, and consider instead
$$
{}^b\sym_{\psi}(A) = \sigma^2 + |\zeta|_{g_Z}^2.
$$
The $b$-principal symbol is oblivious to the potential $V$, which enters $A$ as $x^4$ for large $x$. The point now is to include that term as an extra parameter in the principal symbol, which gives the \emph{extended principal symbol}
$$
\sym_e(A) = \sigma^2 + |\zeta|_{g_Z}^2 + \tau^4,
$$
where $\tau \geq 0$ is an extra covariable that is the receptacle for $x$. The extended principal symbol has mixed homogeneities which are dealt with anisotropically. Note that
$$
\sym_e(A)(\varrho^2\zeta,\varrho^2\sigma,\varrho\tau) = \varrho^4\sym_e(A)(\zeta,\sigma,\tau), \quad \varrho > 0,
$$
so $\sym_e(A)$ is anisotropic homogeneous of order $4$ with respect to the aniso\-tropy vector $\vec{\ell} = (2,2,1) \in \N^3$. The extended principal symbol is thus completely determined by restriction to an anisotropic hemisphere
$$
|\sigma|^2 + |\zeta|_{g_Z}^2 + \tau^4 = 1, \quad \tau \geq 0,
$$
and when further restricting to $\tau = 0$ (the ``equator'' of the hemisphere) we get
$$
\sym_e(A)\big|_{\tau = 0} = {}^b\sym_{\psi}(A)
$$
on the cosphere $|\sigma|^2 + |\zeta|_{g_Z}^2 = 1$. Invertibility of $\sym_e(A)$ for large $x \gg 1$ together with ordinary ellipticity on $M$ constitutes the notion of full ellipticity in our calculus. Clearly $A$ is fully elliptic, and we then find that
\begin{equation}\label{DeltaV}
\Delta_g + V : x^{\alpha}{\mathscr H}^{s;\vec{\ell}}(M) \to x^{\alpha-2}{\mathscr H}^{s-4;\vec{\ell}}(M)
\end{equation}
is Fredholm for all $\alpha,s \in \R$, with kernel and cokernel contained in ${\mathscr S}(M)$ independent of these values. The Sobolev spaces ${\mathscr H}^{s;\vec{\ell}}(M)$ are associated to the pseudodifferential calculus. There is a version of the calculus for any anisotropy vector of the form\footnote{The calculus on the model space $\R_+\times Z$ for the noncompact end allows for more general decoupled anisotropies $\vec{\ell} = (\ell_1,\ell_2,\ell_3) \in \N^3$ that assign different weights to the factors $\R_+$ and $Z$, as well as to the growth in $x$. Such operators are needed elsewhere and are therefore included here, but they are not relevant for the calculus on $M$.} $\vec{\ell} = (\ell_1,\ell_1,\ell_3) \in \N^3$, where in this case $\vec{\ell} = (2,2,1)$. As the first index appears twice, we generally shorten notation to ${\mathscr H}^{s;(\ell_1,\ell_3)}(M)$. The regularity parameter $s \in \R$ controls both smoothness and growth/decay at infinity in the Sobolev space scale, while the weight $\alpha \in \R$ allows for additional fine-tuning of growth. It is always true that $x^{\alpha}{\mathscr H}^{s;\vec{\ell}}(M) \hookrightarrow x^{\alpha'}{\mathscr H}^{s';\vec{\ell}}(M)$ for $s \geq s'$ and $\alpha \leq \alpha'$, and we have
$$
\S(M) = \bigcap\limits_{s \in \R}{\mathscr H}^{s;\vec{\ell}}(M), \quad \S'(M) = \bigcup\limits_{s \in \R}{\mathscr H}^{s;\vec{\ell}}(M).
$$
Moreover,
$$
H^{\frac{s}{\ell_1}}_{\textup{comp}}(M) \subset {\mathscr H}^{s;(\ell_1,\ell_3)}(M)  \subset H^{\frac{s}{\ell_1}}_{\textup{loc}}(M) 
$$
for all $s \in \R$. In the case of the harmonic oscillator we have $\ell_1=2$ and $\ell_3=1$, so smoothness in \eqref{DeltaV} is only shifted by $2$ as expected.

\subsubsection*{Anharmonic oscillator}

Let $V \in C^{\infty}(M)$ be such that $V = x^{2n}$ for large $x \gg 1$, and consider a generalized anharmonic oscillator of the form $(\Delta_g)^m + V$, see \cite{ChatzakouDelgadoRuzhansky} in case $M = \R^d$. Here $m,n \in \N$ are arbitrary. Rewriting $\Delta_g$ in terms of $xD_x$ for large $x$ as above shows that
\begin{align*}
(\Delta_g)^m + V &= x^{-2m}A = x^{-2m}\bigl[\bigl((xD_x)^{2} + \Delta_{g_Z}\bigr)^m + \textup{l.o.t}\bigr] + V \\
&= x^{-2m}\bigl[\bigl((xD_x)^{2} + \Delta_{g_Z}\bigr)^m + x^{2(n+m)} + \textup{l.o.t}\bigr].
\end{align*}
In this case
$$
{}^b\sym_{\psi}(A) = \bigl(\sigma^{2} + |\zeta|_{g_Z}^2\bigr)^m
$$
and
$$
\sym_e(A) = \bigl(\sigma^{2} + |\zeta|_{g_Z}^{2}\bigr)^m + \tau^{2(m+n)},
$$
where again $\tau \geq 0$. Let $\vec{\ell} = (\ell_1,\ell_1,\ell_3) = (m+n,m+n,m) \in \N^3$. Then
$$
\sym_e(A)(\varrho^{\ell_1}\zeta,\varrho^{\ell_1}\sigma,\varrho^{\ell_3}\tau) = \varrho^{2m(m+n)}\sym_e(A)(\zeta,\sigma,\tau), \quad \varrho > 0,
$$
which is of order $2m(m+n)$ with respect to our choice of $\vec{\ell}$. When restricted to $\tau=0$ we again recover ${}^b\sym_{\psi}(A)$, but with a different order convention:
$$
{}^b\sym_{\psi}(A)(\varrho^{\ell_1}\zeta,\varrho^{\ell_1}\sigma) = \varrho^{2m(m+n)}\,{}^b\sym_{\psi}(A)(\zeta,\sigma), \quad \varrho > 0.
$$
While $A$ is of order $2m$ with respect to the standard notion of order, we adopt the point of view that relative to the anisotropy vector $\vec{\ell} =  (m+n,m+n,m)$ the order of $A$ is $2m(m+n)$. As $A$ is elliptic in the usual sense over the compact part of $M$, and $\sym_e(A)$ is invertible for large $x \gg 1$ (for nonzero $(\sigma,\zeta,\tau)$), we find that
$$
(\Delta_g)^m + V : x^{\alpha}{\mathscr H}^{s;(m+n,m)}(M) \to x^{\alpha-2m}{\mathscr H}^{s-2m(m+n);(m+n,m)}(M) 
$$
is Fredholm for all $s,\alpha \in \R$, with kernel and cokernel contained in ${\mathscr S}(M)$. Adding a spectral parameter does not change the result, which confirms as expected that the essential spectrum of $(\Delta_g)^m + V$ is empty, and that eigenfunctions decay rapidly at infinity.

While it would be natural to first simplify the anisotropy vector $\vec{\ell} = (\ell_1,\ell_1,\ell_3) \in \N^3$ by removing any common factors so that  $\gcd(\ell_1,\ell_3) = 1$, we have chosen to allow arbitrary (not necessarily coprime) anisotropy vectors for our analysis.

\subsubsection*{Differential operators in the calculus}

Fix any anisotropy vector $\vec{\ell} = (\ell_1,\ell_1,\ell_3) \in \N^3$. A differential operator in the calculus of order $\mu \in \ell_1\N_0$ is a differential operator $A$ on $M$ of ``usual'' order $\frac{\mu}{\ell_1}$, and for large $x \gg 1$ on the noncompact end, and in coordinates on $Z$, we can write $A$ in the form
$$
A = \sum\limits_{\ell_1|\alpha| + \ell_1 j + \ell_3 k \leq \mu} a_{\alpha,j,k}(x,z)x^kD_z^{\alpha}(xD_x)^j,
$$
where the coefficients $a_{\alpha,j,k}(x,z)$ are $C^{\infty}$, and, for every $i$ and $\beta$, the derivatives $(xD_x)^iD_z^{\beta}a_{\alpha,j,k}(x,z)$ are bounded as $z$ varies in compact subsets of the chart and $x \geq R \gg 1$.

Full ellipticity of $A$ first requires $A$ to be elliptic on $M$ in the usual sense that the homogeneous principal symbol is invertible. We observe that the homogenous principal symbol satisfies
$$
\sym_{\psi}(A)(y,\varrho^{\ell_1}\eta) = \varrho^{\mu}\sym_{\psi}(A), \quad \varrho > 0,
$$
for all $(y,\eta) \in T^*M\setminus 0$, and we consider $A$ to be of anisotropic order $\mu$ with respect to $\vec{\ell}$. For large $x \gg 1$ with $A$ in coordinates written as above we find that
$$
\sym_{\psi}(A)(z,x,\zeta,\xi) = \sum\limits_{\ell_1|\alpha| + \ell_1 j = \mu} a_{\alpha,j,0}(x,z)\zeta^{\alpha}(x\xi)^j,
$$
and after rescaling
$$
{}^b\sym_{\psi}(A)(z,x,\zeta,\sigma) = \sum\limits_{\ell_1|\alpha| + \ell_1 j = \mu} a_{\alpha,j,0}(x,z)\zeta^{\alpha}\sigma^j.
$$
The extended principal symbol is
$$
\sym_e(A)(z,x,\zeta,\sigma,\tau) = \sum\limits_{\ell_1|\alpha| + \ell_1 j + \ell_3 k = \mu} a_{\alpha,j,k}(x,z)\zeta^{\alpha}\tau^k\sigma^j
$$
for $\tau \geq 0$, which is anisotropic homogeneous of order $\mu$ with respect to $\vec{\ell}$. The operator $A$ is fully elliptic if $\sym_e(A)(z,x,\zeta,\sigma,\tau)$ is invertible for all $(z,x)$ with $x \gg 1$ large enough, and all $(\zeta,\sigma,\tau) \neq 0$, and the inverse is bounded when restricted to the anisotropic hemisphere $|\zeta|^{2\ell_3} + |\sigma|^{2\ell_3} + \tau^{2\ell_1} = 1$, $\tau \geq 0$, as $z$ varies over compact sets and $x \geq R \gg 1$. Since we did not compactify the end $x \to \infty$ we need such a boundedness condition on the inverse of $\sym_e(A)$ as $x \to \infty$. In the examples of the harmonic and anharmonic oscillator, the coefficients $a_{\alpha,j,k}$ do not depend on $x$ for large $x$, so this extra boundedness assumption is automatically fulfilled.

If $A$ is fully elliptic it is Fredholm in the associated Sobolev spaces
$$
A : x^{\alpha}{\mathscr H}^{s;\vec{\ell}}(M) \to x^{\alpha}{\mathscr H}^{s-\mu;\vec{\ell}}(M)
$$
for all $s,\alpha \in \R$, with kernel and cokernel contained in $\S(M)$. The presence of extra powers of $x$ as in the harmonic and anharmonic oscillator leads to an additional shift in the weight parameter $\alpha$.

\subsection{{\boldmath Step-$\frac{1}{\ell_1}$} polyhomogeneous pseudodifferential operators}\label{step1lcalc}

Our pseudodifferential calculus associated with $\vec{\ell} = (\ell_1,\ell_1,\ell_3) \in \N^3$ is a refinement of the ordinary step-$\frac{1}{\ell_1}$ polyhomogeneous operator calculus on $M$. The purpose of this section is to set notation, and primarily to explain our point of view regarding the order of operators with respect to $\vec{\ell}$. The manifold $M$ can be any (open) smooth manifold with a positive density here.

By $L^{\mu;\ell_1}(M)$ we denote all operators $A : C_c^{\infty}(M) \to C^{\infty}(M)$ with the following properties:
\begin{itemize}
\item For any $\phi,\psi \in C_c^{\infty}(M)$ with $\textup{supp}(\phi)\cap\textup{supp}(\psi) = \emptyset$, the operator $\phi A \psi$ is an integral operator on $M$ with a $C^{\infty}$-kernel function.
\item Let $\kappa : U \to \Omega$, with $\Omega \subset \R^d$ open, be any local chart, and $\phi,\psi \in C_c^{\infty}(U)$. Then the push-forward $\phi A \psi$ to $\Omega$ is given by
$$
\kappa_*[\phi A \psi] u(y) = (2\pi)^{-d} \int_{\R^d} e^{iy\eta}a(y,\eta)\hat{u}(\eta)\,d\eta,  \quad u \in C_c^{\infty}(\Omega),
$$
where the local symbol $a(y,\eta)$ satisfies the estimates
$$
|D^{\alpha}_y\partial^{\beta}_{\eta}a(y,\eta)| \lesssim \big[(1 + |\eta|^{2})^{\frac{1}{2\ell_1}}\bigr]^{\mu - \ell_1|\beta|}
$$
locally uniformly with respect to $y \in \Omega$.
\end{itemize}
We call $A$ classical and write $A \in L_{\textup{cl}}^{\mu;\ell_1}(M)$ if the local symbols $a(y,\eta)$ have an asymptotic expansion
$$
a(y,\eta) \sim \sum\limits_{j=0}^{\infty}\chi(\eta)a_{(\mu-j)}(y,\eta),
$$
where $\chi \in C^{\infty}(\R^d)$ satisfies $\chi \equiv 0$ near $\eta = 0$ and $\chi \equiv 1$ for large $|\eta|$, and the $a_{(\mu-j)}(y,\eta)$ are anisotropic homogeneous of degree $\mu-j$ with respect to $\ell_1$ in the sense that
$$
a_{(\mu-j)}(y,\varrho^{\ell_1}\eta) = \varrho^{\mu-j}a_{(\mu-j)}(y,\eta), \quad \varrho > 0.
$$
The homogeneous principal symbol $\sym_{\psi}(A)$ of $A$ is defined on $T^*M\setminus 0$ and is (anisotropic) homogeneous of order $\mu$ with respect to $\ell_1$ in the fibers.

Clearly $L^{\mu;\ell_1}(M) = L^{\frac{\mu}{\ell_1}}(M)$, and $L^{\mu;\ell_1}_{\textup{cl}}(M)$ are the step-$\frac{1}{\ell_1}$ polyhomogeneous pseudodifferential operators of usual order $\frac{\mu}{\ell_1}$ on $M$.


\section{The calculus on $\R_+\times Z$}\label{PseudoCalculusExit}

\noindent
We are going to need various types of operators and operator families that depend anisotropically on covariables and parameters. To set notation, let $\R^{N} = \R^{n_1} \times \ldots \times \R^{n_q}$ with $N = n_1 + \ldots +n_q$, and let $\vec{\ell} = (\ell_1,\ldots,\ell_q) \in \N^q$ be arbitrary. Let
\begin{align*}
|y|_{\vec{\ell}} &= \biggl(\sum\limits_{j=1}^q |y_j|^{2\prod_{k \neq j}\ell_k}\biggr)^{\frac{1}{2\ell_1\cdots\ell_q}}, \\
\langle y \rangle_{\vec{\ell}} &= \biggl(1 + \sum\limits_{j=1}^q |y_j|^{2\prod_{k \neq j}\ell_k}\biggr)^{\frac{1}{2\ell_1\cdots\ell_q}}
\end{align*}
for $y = (y_1,\ldots,y_q) \in \R^N$, where $|y_j|$ is the Euclidean norm of $y_j \in \R^{n_j}$ for each $j$. The function $|\cdot|_{\vec{\ell}}$ is $C^{\infty}$ on $\R^N\setminus \{0\}$ and positive, and is anisotropic homogeneous of degree $1$, i.e.,
$$
|(\varrho^{\ell_1}y_1,\ldots,\varrho^{\ell_q}y_q)|_{\vec{\ell}} = \varrho \, |(y_1,\ldots,y_q)|_{\vec{\ell}}, \quad \varrho > 0.
$$
The anisotropic $\vec{\ell}$-spheres $|y|_{\vec{\ell}} = R$ for any $R > 0$ are compact $C^{\infty}$ hypersurfaces of $\R^N$. The anisotropic Japanese bracket $\langle \cdot \rangle_{\vec{\ell}}$ smoothes the singularity at the origin of $|\cdot|_{\vec{\ell}}$. We have Peetre's inequality
$$
\langle y + y' \rangle_{\vec{\ell}}^s \leq 2^{|s|}\langle y \rangle_{\vec{\ell}}^s\,\langle y'\rangle_{\vec{\ell}}^{|s|},
$$
and there exist constants $c,C > 0$ such that
$$
c\langle y\rangle^{\frac{1}{\ell_1+\ldots+\ell_q}} \leq \langle y \rangle_{\vec{\ell}} \leq C \langle y \rangle^{\ell_1+\ldots+\ell_q},
$$
where the standard Japanese bracket $\langle \cdot \rangle$ corresponds to $\vec{\ell} = (1,\ldots,1)$.

\medskip

\noindent
For the remainder of this section, we fix a positive density $dz$ on $Z$. The basic Hilbert space on $\R_+\times Z$ is
$$
L^2_b(\R_+\times Z) = L^2(\R_+\times Z,\tfrac{dx}{x}dz).
$$
For notational brevity we will write out the calculus for operators acting between scalar function spaces, but it equally works for operators acting between sections of vector bundles that are pull-backs to $\R_+\times Z$ of Hermitian vector bundles on $Z$.
 
\subsection{Operator families on {\boldmath $Z$}}\label{sec-ZFamilies}
 
$Z$ is a closed, compact manifold of dimension $\dim Z = f$, where $f=d-1$. Fix $\vec{\ell} = (\ell_1,\ell_2,\ell_3) \in \N^3$. For $\mu \in \R$ let $L^{\mu;\vec{\ell}}(Z;\R\times\overline{\R}_{+})$ denote the parameter-dependent pseudodifferential operator families
$$
A(\sigma,\tau) : C^{\infty}(Z) \to C^{\infty}(Z), \quad (\sigma,\tau) \in \R\times\overline{\R}_+,
$$
of the following kind:
\begin{itemize}
\item For any $\phi,\psi \in C^{\infty}(Z)$ with $\textup{supp}(\phi)\cap\textup{supp}(\psi) = \emptyset$ we have
$$
[\phi A(\sigma,\tau)\psi]u(z) = \int_{Z}k(\sigma,\tau;z,z')u(z')\,dz', \quad u \in C^{\infty}(Z),
$$
with integral kernel
\begin{equation}\label{Zkernels}
k(\sigma,\tau;z,z') \in \S(\R\times\overline{\R}_{+};C^{\infty}(Z\times Z)),
\end{equation}
where $dz'$ is the previously fixed smooth positive density on $Z$.
\item For $\phi,\psi \in C^{\infty}(Z)$ with compact supports in the same local chart $\kappa : U \to \Omega$, the push-forward of $\phi A(\sigma,\tau) \psi$ to $\Omega \subset \R^f$ is given by
$$
\kappa_*[\phi A(\sigma,\tau) \psi] u(z) = (2\pi)^{-f} \int_{\R^f} e^{iz\zeta}a(z,\zeta;\sigma,\tau)\hat{u}(\zeta)\,d\zeta,  \quad u \in C_c^{\infty}(\Omega),
$$
where the symbol $a(z,\zeta;\sigma,\tau) \in C^{\infty}(\Omega\times\R^f\times\R\times\overline{\R}_{+})$ satisfies the estimates
\begin{equation}\label{symbolestimate}
|D^{\alpha}_z\partial^{\beta}_{\zeta}\partial^{\gamma}_{\sigma}\partial^{\delta}_{\tau}a(z,\zeta;\sigma,\tau)| \lesssim \langle (\zeta,\sigma,\tau) \rangle_{\vec{\ell}}^{\mu-\ell_1|\beta|-\ell_2|\gamma|-\ell_3|\delta|}
\end{equation}
locally uniformly with respect to $z \in \Omega$.
\end{itemize}
The regularizing families
$$
L^{-\infty}(Z;\R\times\overline{\R}_{+}) = \bigcap_{\mu \in \R} L^{\mu;\vec{\ell}}(Z;\R\times\overline{\R}_{+})
$$
are independent of the anisotropy vector $\vec{\ell}$ and consist of operator families with integral kernels \eqref{Zkernels}.

We write $L_{\textup{cl}}^{\mu;\vec{\ell}}(Z;\R\times\overline{\R}_{+})$ for the subclass of classical operator families, i.e., the local symbols $a(z,\zeta;\sigma,\tau)$ are classical symbols in the sense that
$$
a(z,\zeta;\sigma,\tau) \sim \sum\limits_{j=0}^{\infty}\chi(\zeta;\sigma,\tau)a_{(\mu-j)}(z,\zeta;\sigma,\tau),
$$
where $\chi \in C^{\infty}(\R^f\times\R\times\overline{\R}_{+})$ is an excision function of the origin, and the $a_{(\mu-j)} \in C^{\infty}(\Omega\times[\bigl(\R^f
\times\R\times\overline{\R}_{+}\bigr)\setminus 0])$ are anisotropic homogeneous of degree $\mu-j$:
$$
a_{(\mu-j)}(z,\varrho^{\ell_1}\zeta;\varrho^{\ell_2}\sigma,\varrho^{\ell_3}\tau) = \varrho^{\mu-j}a_{(k)}(z,\zeta;\sigma,\tau), \quad \varrho > 0.
$$
Every $A(\sigma,\tau) \in L_{\textup{cl}}^{\mu;\vec{\ell}}(Z;\R\times\overline{\R}_{+})$ has an invariantly defined parameter-dependent homogeneous principal symbol
$\sym(A)(z,\zeta;\sigma,\tau)$ on $(T^*Z\times\R\times\overline{\R}_{+})\setminus 0$ that satisfies
$$
\sym(A)(z,\varrho^{\ell_1}\zeta,\varrho^{\ell_2}\sigma,\varrho^{\ell_3}\tau) = \varrho^{\mu}\sym(A)(z,\zeta;\sigma,\tau), \quad \varrho > 0.
$$
We equip both $L^{\mu;\vec{\ell}}(Z;\R\times\overline{\R}_{+})$ and $L_{\textup{cl}}^{\mu;\vec{\ell}}(Z;\R\times\overline{\R}_{+})$ with their canonical Fr{\'e}chet topologies. We reiterate that for the global calculus on the noncompact manifold only $\ell_1=\ell_2$ is relevant.

\subsubsection*{Properties of the calculus}

The operator calculus has the expected properties, which follows with the standard proofs. Composition of operators yields
$$
L_{\textup{(cl)}}^{\mu_1;\vec{\ell}}(Z;\R\times\overline{\R}_{+}) \times L_{\textup{(cl)}}^{\mu_2;\vec{\ell}}(Z;\R\times\overline{\R}_{+}) \to L_{\textup{(cl)}}^{\mu_1+\mu_2;\vec{\ell}}(Z;\R\times\overline{\R}_{+}),
$$
and in the classical case $\sym(AB) = \sym(A)\sym(B)$.

Likewise, the formal adjoint with respect to the $L^2$-inner product on $Z$ gives a map
$$
L_{\textup{(cl)}}^{\mu;\vec{\ell}}(Z;\R\times\overline{\R}_{+}) \ni A(\sigma,\tau) \mapsto A^*(\sigma,\tau) \in L_{\textup{(cl)}}^{\mu;\vec{\ell}}(Z;\R\times\overline{\R}_{+}),
$$
and in the classical case $\sym(A^*) = \sym(A)^*$.

Parametrices to parameter-dependent elliptic elements exist and are constructed in the usual way. Here $A(\sigma,\tau) \in L_{\textup{cl}}^{\mu;\vec{\ell}}(Z;\R\times\overline{\R}_{+})$ is parameter-dependent elliptic if $\sym(A)$ is invertible on $(T^*Z\times\R\times\overline{\R}_{+})\setminus 0$. This is equivalent to the existence of a parameter-dependent parametrix $B(\sigma,\tau) \in L_{\textup{cl}}^{-\mu;\vec{\ell}}(Z;\R\times\overline{\R}_{+})$ such that both
$$
A(\sigma,\tau)B(\sigma,\tau) - I,\, B(\sigma,\tau)A(\sigma,\tau) - I \in L^{-\infty}(Z;\R\times\overline{\R}_{+}).
$$

\subsubsection*{Holomorphic families and kernel cut-off}

We need operator families
$$
A(\sigma,\tau) : C^{\infty}(Z) \to C^{\infty}(Z)
$$
that depend holomorphically on $\sigma \in \C$ in the entire complex plane. We write $M_{O,\textup{(cl)}}^{\mu;\vec{\ell}}(Z;\overline{\R}_+)$ to denote such holomorphic families so that, when restricted to lines parallel to the real axis, the function
$$
\R \ni \gamma \mapsto A(\sigma + i\gamma,\tau) \in L_{\textup{(cl)}}^{\mu;\vec{\ell}}(Z;\R\times\overline{\R}_+), \quad (\sigma,\tau) \in \R\times\overline{\R}_+,
$$
is locally bounded.

Restricting to different lines $\Im(\sigma) = \gamma$ in the complex plane gives families that are related by asymptotic expansions
$$
A(\sigma+i\gamma,\tau) \sim \sum\limits_{j=0}^{\infty}\frac{(i\gamma)^j}{j!}\partial^j_\sigma A(\sigma,\tau)
$$
in $L_{\textup{(cl)}}^{\mu;\vec{\ell}}(Z;\R\times\overline{\R}_+)$, which hold locally uniformly with respect to $\gamma \in \R$. In particular,
$$
M_{O}^{\mu;\vec{\ell}}(Z;\overline{\R}_+)\cap L^{\mu';\vec{\ell}}(Z;\{\Im(\sigma)=\gamma\}\times\overline{\R}_+) = M_{O}^{\mu';\vec{\ell}}(Z;\overline{\R}_+)
$$
for $\mu' \leq \mu$ and any $\gamma \in \R$. Restriction to the real line thus gives a well-defined map
$$
M_{O,\textup{(cl)}}^{\mu;\vec{\ell}}(Z;\overline{\R}_+)/M_{O}^{-\infty}(Z;\overline{\R}_+) \to L_{\textup{(cl)}}^{\mu;\vec{\ell}}(Z;\R\times\overline{\R}_+)/L^{-\infty}(Z;\R\times\overline{\R}_+).
$$
This map is an isomorphism. Its inverse is given by the kernel cut-off operator
\begin{equation}\label{HphiSpaces}
H(\phi) : L_{\textup{(cl)}}^{\mu;\vec{\ell}}(Z;\R\times\overline{\R}_+) \to M_{O,\textup{(cl)}}^{\mu;\vec{\ell}}(Z;\overline{\R}_+).
\end{equation}
Here $\phi \in C_c^{\infty}(\R)$ is arbitrary with $\phi \equiv 1$ near zero, and the kernel cut-off operator is defined as $H(\phi)A = {\mathcal F}_{t \to \sigma}\phi(t){\mathcal F}^{-1}_{\sigma \to t}A$, which can be written as an oscillatory integral
\begin{equation}\label{kernelcutoff}
[H(\phi)A](\sigma+i\gamma,\tau) = \frac{1}{2\pi} \iint e^{-its}e^{t\gamma}\phi(t)A(\sigma-s,\tau)\,dt ds \\
\end{equation}
for $\sigma + i\gamma \in \C$ and $\tau \geq 0$. The standard regularization procedure applied to this integral shows that $H(\phi)A \in M_{O,\textup{(cl)}}^{\mu;\vec{\ell}}(Z;\overline{\R}_+)$ (the analyticity follows by direct verification of the Cauchy-Riemann equations). Moreover,
\begin{align*}
A(\sigma,\tau) - [H(\phi)A](\sigma,\tau) &= \frac{1}{2\pi} \iint e^{-its}(1-\phi(t))A(\sigma-s,\tau)\,dt ds \\
&= \frac{i^k}{2\pi} \iint e^{-its}\frac{1-\phi(t)}{t^k}[\partial^k_{\sigma}A](\sigma-s,\tau)\,dt ds
\end{align*}
for any $k \in \N_0$, and so $A(\sigma,\tau) - [H(\phi)A](\sigma,\tau) \in L^{-\infty}(Z;\R\times\overline{\R}_+)$. The kernel cut-off operator $H(\phi)$ is a continuous map between the Fr{\'e}chet spaces \eqref{HphiSpaces}.

Holomorphic families of parameter-dependent pseudodifferential operators which are then utilized as operator-valued symbols and the kernel cut-off construction originate from Schulze's theory of pseudodifferential operators on manifolds with conical and more general singularities \cite{SchuNH}. The oscillatory integral representation of the kernel cut-off operator and detailed proofs of the above claims can be found, e.g., in \cite[Section~3]{KrainerVolterra} (for a different symbol class, but the arguments carry over as stated there).

\begin{lemma}\label{AnalyticSmoothingApprox}
Let $A \in M^{\mu;\vec{\ell}}_O(Z;\overline{\R}_+)$. Then there exist $A_j \in M^{-\infty}_O(Z;\overline{\R}_+)$ such that $A_j \to A$ in $M^{\mu';\vec{\ell}}_O(Z;\overline{\R}_+)$ as $j \to \infty$ for every $\mu' > \mu$. The $A_j \in M^{-\infty}_O(Z;\overline{\R}_+)$ can be explicitly constructed to depend linearly and continuously on $A \in M^{\mu;\vec{\ell}}_O(Z;\overline{\R}_+)$.
\end{lemma}
\begin{proof}
Let $\phi_k$ be a partition of unity subordinate to some finite covering of $Z$ by coordinate charts, and let $\psi_k$ be supported in the same coordinate neighborhood as $\phi_k$ with $\psi_k \equiv 1$ in a neighborhood of the support of $\phi_k$, $k=1,\ldots,N$. Write
$$
A (\sigma,\tau) = \sum\limits_{k=1}^N \phi_kA(\sigma,\tau) \psi_k+ R(\sigma,\tau),
$$
where $R(\sigma,\tau) \in M^{-\infty}_O(Z;\overline{\R}_+)$. Each operator $\phi_kA(\sigma,\tau)\psi_k$ is the pull-back of a local operator from an open set in $\R^f$ with a symbol $a_k(z,\zeta;\sigma,\tau)$ that is holomorphic in $\sigma \in \C$ and satisfies the estimates \eqref{symbolestimate} when restricted to lines $\{\sigma + i\gamma;\;\sigma \in \R\}$, and these estimates hold locally uniformly with respect to the imaginary part $\gamma \in \R$. The claim of the lemma thus reduces to the corresponding statement for such local symbols $a(z,\zeta;\sigma,\tau)$. For brevity, we denote the real anisotropic local symbol class that is based on the estimates \eqref{symbolestimate} by $S^{\mu;\vec{\ell}}$, and the holomorphic symbols by $S_O^{\mu;\vec{\ell}}$.

Let $a(z,\zeta;\sigma,\tau) \in S_O^{\mu;\vec{\ell}}$ be arbitrary, and let $\chi \in C^{\infty}(\R^f\times\R\times\R)$ be such that $\chi \equiv 0$ for $|(\zeta,\sigma,\tau)|_{\vec{\ell}} \leq 1$ and $\chi \equiv 1$ for $|(\zeta,\sigma,\tau)|_{\vec{\ell}} \geq 2$, and let
$$
b_j(z,\zeta;\sigma,\tau) = \chi\bigl(\tfrac{\zeta}{j^{\ell_1}},\tfrac{\sigma}{j^{\ell_2}},\tfrac{\tau}{j^{\ell_3}}\bigr)a(z,\zeta;\sigma,\tau) \in S^{\mu;\vec{\ell}}, \; j \in \N,
$$
where $a$ is restricted to $\sigma \in \R$. Then $b_j \to 0$ as $j \to \infty$ in $S^{\mu';\vec{\ell}}$ for every $\mu' > \mu$, and $a - b_j \in S^{-\infty}$. Now let $c_j = H(\phi)b_j \in S_O^{\mu;\vec{\ell}}$ with the kernel cut-off operator \eqref{kernelcutoff} (at the level of local symbols, $z$ and $\zeta$ are extra parameters in the formula). Then $b_j - c_j \in S^{-\infty}$, and by the continuity of the kernel cut-off operator we have $c_j \to 0$ in $S^{\mu';\vec{\ell}}_O$ for $\mu' > \mu$. The lemma then follows with
$$
a_j = a - c_j \in S^{\mu;\vec{\ell}}_O \cap S^{-\infty} = S^{-\infty}_O.
$$
\end{proof}

\subsection{Residual operators}\label{sec-ResidualModel}
 
Let $H^s_b = H^{s}_b(\R_+\times Z)$ be the standard $b$-Sobolev space of smoothness $s \in \R$ based on $L^2_b = L^2_b(\R_+\times Z)$ on $\R_+\times Z$. Then
$$
\S_0(\R_+\times Z) = \bigcap_{s,\alpha \in \R} x^{\alpha}H^s_b(\R_+\times Z),
$$
where $\S_0(\R_+\times Z)$ is the space of $C^{\infty}$ functions that are rapidly decreasing as $x \to \infty$ and vanish to infinite order at $x=0$.

Let $\Psi_G^{-\infty}(\R_+\times Z)$ denote the space of all continuous operators
$$
G : \S_0(\R_+\times Z) \to \S_0(\R_+\times Z)
$$
with the following properties:
\begin{itemize}
\item $G$ has a formal adjoint $G^*$ with respect to the $L^2_b$-inner product, and both $G$ and $G^*$ are continuous maps
$$
G,\,G^* : \S_0(\R_+\times Z) \to \S_0(\R_+\times Z).
$$
\item For every $\alpha \in \R$, both $G$ and $G^*$ extend by continuity to continuous operators
$$
G,\,G^* : x^{\alpha}H^s_b(\R_+\times Z) \to x^{\alpha'}H_b^{s'}(\R_+\times Z)
$$
for all $s,s' \in \R$, and every $\alpha' \leq \alpha$.
\end{itemize}

\begin{lemma}\label{ResGreen}
Let $G \in \Psi^{-\infty}_G(\R_+\times Z)$ be arbitrary.
\begin{enumerate}[(a)]
\item $x^{-\beta}Gx^{\beta} \in \Psi^{-\infty}_G(\R_+\times Z)$ for all $\beta \in \R$.
\item Let $\omega \in C_c^{\infty}(\overline{\R}_+)$ be such that $\omega \equiv 1$ near $x = 0$. Then
$$
(1-\omega)G,\,G(1-\omega) : x^{\alpha}H^s_b \to x^{\alpha'}H^{s'}_b
$$
for all $s,s',\alpha,\alpha' \in \R$.
\end{enumerate}
\end{lemma}
\begin{proof}
Part (a) follows at once from the mapping properties in the definition. For (b) we note that
$$
x^{-\beta}(1-\omega) : x^{\alpha'}H^{s'}_b \to x^{\alpha'}H^{s'}_b
$$
is continuous for all $\beta \geq 0$, and so
$$
(1-\omega) G = x^{\beta} [x^{-\beta}(1-\omega)] G : x^{\alpha}H^s_b \to x^{\alpha'+\beta}H^{s'}_b
$$
for all $s,s' \in \R$ and all $\alpha' \leq \alpha$ and $\beta \geq 0$. As $\beta \geq 0$ is arbitrary we obtain the asserted mapping properties for $(1-\omega)G$. The claim for $G(1-\omega)$ then follows by considering the adjoint, or by arguing as above by writing
$$
G(1-\omega) = x^{\beta} [x^{-\beta} G x^{\beta}] [x^{-\beta}(1-\omega)]
$$
for arbitary $\beta \geq 0$ and using (a).
\end{proof}

\noindent
By Lemma~\ref{ResGreen}, the operators $(1-\omega)G$ and $G(1-\omega)$ are integral operators of the form
$$
\int_{\R_+\times Z} g(x,z,x',z')u(x',z')\,\frac{dx'}{x'}\,dz'
$$
with kernels
$$
g(x,z,x',z') \in \S_0([\R_+\times Z] \times [\R_+\times Z]) = \S_0(\R_+\times Z) \hat{\otimes}_{\pi} \S_0(\R_+\times Z).
$$
The kernels of residual operators $G$ thus exhibit rapid decay as $x,\,x' \to \infty$. We do not control their behavior as $x \to 0$ and $x' \to 0$ beyond the mapping properties stated in the definition of these operators as our focus lies on the behavior at infinity.

\subsection{Mellin pseudodifferential operators on {\boldmath $\R_+\times Z$}}\label{sec-MellinCalcReview}

We review some aspects of the standard Mellin pseudodifferential calculus on $\R_+\times Z$, see \cite{GilSchulzeSeiler0,GilSchulzeSeiler,KrainerMellin,SchuNH}.

Let $L^m(Z;\R)$ be the usual (isotropic) space of pseudodifferential operators of order $m \in \R$ on $Z$ that depend on the parameter $\sigma \in \R$. We consider an operator-valued symbol $p(x,\sigma) \in C^{\infty}(\R_+,L^m(Z;\R))$ and associate with it the Mellin pseudodifferential operator
$$
\opm(p) : C_c^{\infty}(\R_+\times Z) \to C^{\infty}(\R_+\times Z)
$$
given by
$$
[\opm(p)u](x) = \frac{1}{2\pi}\int_{\R}\int_0^{\infty}\Bigl(\frac{x}{x'}\Bigr)^{i\sigma}p(x,\sigma)u(x')\,\frac{dx'}{x'}\,d\sigma.
$$
We use here the convention
\begin{align*}
\bigr(M u\bigl)(\sigma) &= \int_0^{\infty}x^{-i\sigma}u(x)\,\frac{dx}{x}, \quad \sigma \in \R, \\
\bigr(M^{-1}v\bigl)(x) &= \frac{1}{2\pi}\int_{\R}x^{i\sigma}v(\sigma)\,d\sigma, \quad x > 0,
\end{align*}
for the Mellin transform and its inverse. If
$$
p \in C^{\infty}_B(\R_+,L^m(Z;\R)),
$$
where $C^{\infty}_B(\R_+,F)$ stands for the $C^{\infty}$-functions $g$ on $\R_+$ with values in the Fr{\'e}chet space $F$ such that all derivatives $(xD_x)^jg$ are bounded on $\R_+$, then
$$
\opm(p) : H^s_b(\R_+\times Z) \to H^{s-m}_b(\R_+\times Z)
$$
is continuous for all $s \in \R$. If the operator-valued symbol extends with respect to $\sigma$ to an entire function such that
$$
p(x,\sigma)\big|_{\Im(\sigma) = \gamma} \in C^{\infty}_B(\R_+,L^m(Z;\R))
$$
locally uniformly with respect to $\gamma \in \R$, i.e., if
$$
p(x,\sigma) \in C^{\infty}_B(\R_+,M^m_O(Z)),
$$
then
$$
\opm(p) : x^{\alpha}H^s_b(\R_+\times Z) \to x^{\alpha}H^{s-m}_b(\R_+\times Z)
$$
is continuous for all $s,\alpha \in \R$, and gives rise to a continuous operator
\begin{equation}\label{MellinS0}
\opm(p) :  \S_0(\R_+ \times Z) \to \S_0(\R_+ \times Z).
\end{equation}
This is not true without analyticity in the covariable $\sigma$, and is based on
\begin{equation}\label{conjugate}
x^{-\beta} \opm(p) x^{\beta} = \opm(p_{\beta}) : C_c^{\infty}(\R_+\times Z) \to C^{\infty}(\R_+\times Z),
\end{equation}
where $p_{\beta}(x,\sigma) = p(x,\sigma - i\beta)$ for every $\beta \in \R$.

Standard Mellin pseudodifferential operators with holomorphic symbols form a well-behaved calculus of operators acting in the spaces \eqref{MellinS0} (and extended to larger spaces, such as the weighted $H^s_b$-spaces, by continuity). One of the features of the calculus is the availability of explicit oscillatory integral formulas for analyzing formal adjoints and compositions of operators. We recall the statements and formulas.

\subsubsection*{Formal adjoints}

Let $p(x,\sigma) \in C^{\infty}_B(\R_+,M^{m}_O(Z))$. Then the formal adjoint of $\opm(p)$ with respect to the $L^2_b$-inner product on $\R_+\times Z$ is given by
$$
[\opm(p)]^* = \opm(p^{\star}) : \S_0(\R_+ \times Z) \to \S_0(\R_+ \times Z),
$$
where $p^{\star} \in C^{\infty}_B(\R_+,M^{m}_O(Z))$. For every $N \in \N$
\begin{equation}\label{MelAdjoint}
\begin{aligned}
p^{\star}(x,\sigma) &= \frac{1}{2\pi}\iint y^{-i\eta}p(xy,\overline{\sigma}+\eta)^*\,\frac{dy}{y}d\eta \\
&= \sum\limits_{k=0}^{N-1}\frac{1}{k!} [(-xD_x)^k\partial_{\sigma}^kp](x,\overline{\sigma})^* + r_N(x,\sigma)
\end{aligned}
\end{equation}
with
\begin{equation}\label{MelAdjointRem}
r_N(x,\sigma) = \frac{1}{2\pi}\!\int_0^1\frac{(1-\theta)^{N-1}}{(N-1)!}\iint y^{-i\eta}[(-xD_x)^{N}\partial_{\sigma}^Np](xy,\overline{\sigma}+\theta\eta)^*\,\frac{dy}{y}d\eta d\theta.
\end{equation}
We have $r_N(x,\sigma) \in C^{\infty}_B(\R_+,M^{m-N}_O(Z))$. The adjoints that appear with the operator-valued symbols in these formulas refer to the $L^2$-inner product on $Z$.

\subsubsection*{Composition}

Let $p_j(x,\sigma) \in C^{\infty}_B(\R_+,M^{m_j}_O(Z))$. Then the composition
$$
\opm(p_1) \, \opm(p_2) = \opm(p_1{\#}p_2) : \S_0(\R_+ \times Z) \to \S_0(\R_+ \times Z)
$$
with $p_1{\#}p_2 \in C^{\infty}_B(\R_+,M^{m_1+m_2}_O(Z))$. For every $N \in \N$
\begin{equation}\label{MelLeibniz}
\begin{aligned}
p_1{\#}p_2(x,\sigma) &= \frac{1}{2\pi}\iint y^{-i\eta}p_1(x,\sigma+\eta)p_2(xy,\sigma)\,\frac{dy}{y}d\eta \\
&= \sum\limits_{k=0}^{N-1}\frac{1}{k!}[\partial_{\sigma}^kp_1](x,\sigma)[(xD_x)^kp_2](x,\sigma) + r_N(x,\sigma)
\end{aligned}
\end{equation}
with
\begin{equation}\label{MelLeibnizRem}
r_N(x,\sigma) = \frac{1}{2\pi}\!\int_0^1\!\frac{(1-\theta)^{N-1}}{(N-1)!}\iint y^{-i\eta}[\partial_{\sigma}^Np_1](x,\sigma+\theta\eta)[(xD_{x})^Np_2](xy,\sigma)\,\frac{dy}{y}d\eta d\theta.
\end{equation}
We have $r_N \in C^{\infty}_B(\R_+,M^{m_1+m_2-N}_O(Z))$.

\subsubsection*{Anisotropic Mellin calculus}

The anisotropic operator calculus on $\R_+ \times Z$ is based on an anisotropy vector $\vec{\ell} = (\ell_1,\ell_2,\ell_3) \in \N^3$. The first two parameters $\ell_1$ and $\ell_2$ are associated with $Z$ and $\R_+$, respectively. Only the case $\ell_1=\ell_2$ is relevant for building the final calculus on $M$; however, the case $\ell_1 \neq \ell_2$ on $\R_+\times Z$ is needed in other contexts. If $\ell_1 \neq \ell_2$ the isotropic Mellin calculus reviewed above does not suffice, and we need an anisotropic version.

The starting point for the anisotropic Mellin calculus is the parameter-dependent class $L^{m;(\ell_1,\ell_2)}(Z;\R)$. The local symbols in coordinate patches for $Z$ for these operator families are $a(z,\zeta;\sigma)$ such that
$$
|D_z^{\alpha}\partial_{\zeta}^{\beta}\partial_{\sigma}^{\gamma}a(z,\zeta;\sigma)| \lesssim \langle (\zeta,\sigma) \rangle_{(\ell_1,\ell_2)}^{m-\ell_1|\beta|-\ell_2|\gamma|}.
$$
This is similar to, but simpler than, the estimates \eqref{symbolestimate} for the operator families discussed in Section~\ref{sec-ZFamilies}. The associated holomorphic class is $M_O^{m;(\ell_1,\ell_2)}(Z)$. Anisotropic Mellin pseudodifferential operators are based on operator-valued symbols
$$
p(x,\sigma) \in C^{\infty}_B(\R_+,M_O^{m;(\ell_1,\ell_2)}(Z)),
$$
and the operators $\opm(p)$ again act between the spaces \eqref{MellinS0}. The statements and formulas for formal adjoints and compositions above remain valid when the isotropic class of holomorphic symbols is consistently replaced by $M_O^{m;(\ell_1,\ell_2)}(Z)$.

What changes in the anisotropic situation are the $b$-Sobolev spaces on $\R_+\times Z$. For $s \in \R$ pick any $R(\sigma) \in L^{s;(\ell_1,\ell_2)}(Z;\R)$ that is invertible with $R^{-1}(\sigma) \in L^{-s;(\ell_1,\ell_2)}(Z;\R)$, and define
\begin{gather*}
H_b^{s;(\ell_1,\ell_2)}(\R_+\times Z) = \textup{Closure of $C_c^{\infty}(\R_+\times Z)$ with respect to the norm} \\
\|u\|_{H_b^{s;(\ell_1,\ell_2)}} = \|\opm(R)u\|_{L^2_b}.
\end{gather*}
This space does not depend on the choice of $R(\sigma)$, is localizable in coordinate patches on $Z$ (as a mixed anisotropic Mellin-Fourier Sobolev space), and, for $\ell_1=\ell_2 = 1$, reproduces the standard (isotropic) $b$-Sobolev space of smoothness $s \in \R$. By localizing we furthermore obtain the inclusions
\begin{equation}\label{HsIncl}
\begin{gathered}
H_b^{s(\ell_1+\ell_2)}(\R_+\times Z) \hookrightarrow H_b^{s;(\ell_1,\ell_2)}(\R_+\times Z) \hookrightarrow H_b^{\frac{s}{\ell_1+\ell_2}}(\R_+\times Z), \quad s \geq 0, \\
H_b^{\frac{s}{\ell_1+\ell_2}}(\R_+\times Z) \hookrightarrow H_b^{s;(\ell_1,\ell_2)}(\R_+\times Z) \hookrightarrow H_b^{s(\ell_1+\ell_2)}(\R_+\times Z),
\quad s < 0.
\end{gathered}
\end{equation}
Reductions of order as used in the definition of the spaces can be constructed by adding an extra parameter $\lambda \in \R$ to the calculus, and then choosing a parameter-dependent elliptic operator $R(\sigma,\lambda) \in L_{\textup{cl}}^{s;(\ell_1,\ell_2,\ell_1)}(Z;\R\times\R)$ (quantized from an invertible principal symbol on $[T^*Z\times\R\times\R]\setminus 0$, e.g., $|(\zeta,\sigma,\lambda)|_{(\ell_1,\ell_2,\ell_1)}^s$, where an arbitrary Riemannian metric on $Z$ is used to measure the lengths of covectors $\zeta \in T^*Z$). Then there exists a parameter-dependent parametrix $\tilde{R}(\sigma,\lambda) \in L_{\textup{cl}}^{-s;(\ell_1,\ell_2,\ell_1)}(Z;\R\times\R)$, sufficiently refined so as to satisfy
$$
R(\sigma,\lambda)\tilde{R}(\sigma,\lambda) - I,\, \tilde{R}(\sigma,\lambda)R(\sigma,\lambda) - I \in C_c^{\infty}(\R\times\R,L^{-\infty}(Z)).
$$
For $|\lambda_0| \gg 0$ large enough, $R(\sigma):= R(\sigma,\lambda_0) \in L^{s;(\ell_1,\ell_2)}(Z;\R)$ is then invertible with inverse $R^{-1}(\sigma) = \tilde{R}(\sigma,\lambda_0) \in L^{-s;(\ell_1,\ell_2)}(Z;\R)$.

Anisotropic Mellin operators $\opm(p)$ with $p(x,\sigma) \in C^{\infty}_B(\R_+,M_O^{m;(\ell_1,\ell_2)}(Z))$ are continuous in the anisotropic $b$-Sobolev spaces
$$
\opm(p) : x^{\alpha}H_b^{s;(\ell_1,\ell_2)}(\R_+\times Z) \to x^{\alpha}H_b^{s-m;(\ell_1,\ell_2)}(\R_+\times Z)
$$
for all $s,\alpha \in \R$.

The definition of the residual class $\Psi^{-\infty}_G(\R_+\times Z)$ in Section~\ref{sec-ResidualModel} is based on mapping properties in the scale of standard $b$-Sobolev spaces with weights. The inclusions \eqref{HsIncl} between the anisotropic and isotropic $b$-Sobolev spaces show that we obtain the same residual operator class if we instead base the definition on the analogous mapping properties in anisotropic weighted spaces $x^{\alpha}H_b^{s;(\ell_1,\ell_2)}(\R_+\times Z)$.

\subsection{Definition of the calculus on {\boldmath $\R_+\times Z$}}

The class $\Psi^{\mu;\vec{\ell}}_{\textup{(cl)}}(\R_+\times Z)$ of (classical) pseudodifferential operators of order $\mu \in \R$ relative to the given anisotropy vector $\vec{\ell} = (\ell_1,\ell_2,\ell_3) \in \N^3$ consists of operators of the form
\begin{equation}\label{AModel}
A = \opm(p) + G : \S_0(\R_+\times Z) \to \S_0(\R_+ \times Z)
\end{equation}
with $G \in \Psi^{-\infty}_G(\R_+\times Z)$, and $p(x,\sigma)$ is an operator-valued holomorphic Mellin symbol of the form
$$
p(x,\sigma) = a(x,\sigma,x)
$$
for some
\begin{equation}\label{asymbol}
a(x,\sigma,\tau) \in C_B^{\infty}(\R_+,M_{O,\textup{(cl)}}^{\mu;\vec{\ell}}(Z;\overline{\R}_+)).
\end{equation}
Note that the variable $x$ enters twice in $p(x,\sigma)$: Once as a parameter, which is the crucial feature of the calculus, and once as an extra variable to allow for additional dependence on $x$.

\begin{enumerate}
\item For the analysis of operators it is relevant that the scalar function
$$
\tau \in M_{O,\textup{cl}}^{\ell_3;\vec{\ell}}(Z;\overline{\R}_+),
$$
which shows that the operator of multiplication by $x$ belongs to $\Psi^{\ell_3;\vec{\ell}}_{\textup{cl}}(\R_+\times Z)$. More generally, $x^j \in \Psi^{j\ell_3;\vec{\ell}}_{\textup{cl}}(\R_+\times Z)$ for $j \in \N$.

\item Observe that
\begin{align*}
(xD_x)p(x,\sigma) &= (xD_x a)(x,\sigma,x) + (\tau D_{\tau}a)(x,\sigma,x) \\
&= [(xD_x  + \tau D_{\tau}) a](x,\sigma,x),
\end{align*}
and
$$
(xD_x  + \tau D_{\tau}) a \in C_B^{\infty}(\R_+,M_{O,\textup{(cl)}}^{\mu;\vec{\ell}}(Z;\overline{\R}_+)).
$$
Thus
$$
(xD_x)^j\partial_{\sigma}^k p(x,\sigma) = [(xD_x  + \tau D_{\tau})^j\partial_{\sigma}^k a](x,\sigma,x)
$$
with
$$
(xD_x  + \tau D_{\tau})^j\partial_{\sigma}^k a \in C_B^{\infty}(\R_+,M_{O,\textup{(cl)}}^{\mu-k\ell_2;\vec{\ell}}(Z;\overline{\R}_+))
$$
for all $j,k \in \N_0$. In particular,
$$
\langle x \rangle^{-\frac{\mu_+}{\ell_3}}p(x,\sigma) \in C_B^{\infty}(\R_+,M_O^{\mu;(\ell_1,\ell_2)}(Z)),
$$
where $\mu_+ = \max\{\mu,0\}$, which shows that
$$
\opm(p) : \S_0(\R_+\times Z) \to \S_0(\R_+ \times Z)
$$
as is implicitly claimed in \eqref{AModel}.

\item The map $(a,u) \mapsto \opm(p)u$ is continuous in
$$
C_B^{\infty}(\R_+,M_{O}^{\mu;\vec{\ell}}(Z;\overline{\R}_+)) \times \S_0(\R_+ \times Z) \to \S_0(\R_+ \times Z).
$$
Thus, if $a_j \in C_B^{\infty}(\R_+,M_{O}^{-\infty}(Z;\overline{\R}_+))$ with $a_j \to a \in C_B^{\infty}(\R_+,M_{O}^{\mu';\vec{\ell}}(Z;\overline{\R}_+))$ as $j \to \infty$ for $\mu' > \mu$, then $\opm(p_j)u \to \opm(p)u$ in $\S_0(\R_+ \times Z)$ as $j \to \infty$ for every $u \in \S_0(\R_+\times Z)$, where $p_j(x,\sigma) = a_j(x,\sigma,x)$.

The existence of such approximating sequences $a_j$ for every operator-valued symbol $a \in C_B^{\infty}(\R_+,M_{O}^{\mu;\vec{\ell}}(Z;\overline{\R}_+))$ follows from Lemma~\ref{AnalyticSmoothingApprox}. The operators $\opm(p_j)$ then are all standard (smoothing) Mellin pseudodifferential operators with holomorphic symbols, and the properties of Mellin pseudodifferential calculus are applicable to these operators, in particular the oscillatory integral formulas for compositions and adjoints as reviewed in Section~\ref{sec-MellinCalcReview}. Such approximation arguments are used to prove that compositions and formal adjoints are well-behaved for the class $\Psi^{\ast;\vec{\ell}}_{\textup{(cl)}}(\R_+\times Z)$.
\end{enumerate}

\subsubsection*{The extended principal symbol}

Taylor expansion with respect to $\tau$ in \eqref{asymbol} gives an asymptotic expansion
$$
a(x,\sigma,\tau) \sim \sum\limits_{j=0}^{\infty} \tau^j a_j(x,\sigma),
$$
viewed merely as an expansion in $C^{\infty}(\R_+,L_{\textup{(cl)}}^{\mu;(\ell_1,\ell_2)}(Z;\R))$ that depends smoothly on the parameter $\tau \in \R_+$, where
$$
a_j(x,\sigma) = \frac{1}{j!}\partial_{\tau}^ja(x,\sigma,0) \in C^{\infty}(\R_+,L_{\textup{(cl)}}^{\mu-j\ell_3;(\ell_1,\ell_2)}(Z;\R)).
$$
In particular,
\begin{equation}\label{pxstay}
p(x,\sigma) \sim \sum\limits_{j=0}^{\infty} x^j a_j(x,\sigma)
\end{equation}
in $C^{\infty}(\R_+,L_{\textup{(cl)}}^{\mu;(\ell_1,\ell_2)}(Z;\R))$. If $a$ is classical, then $p(x,\sigma) \in C^{\infty}(\R_+,L_{\textup{cl}}^{\mu;(\ell_1,\ell_2)}(Z;\R))$ is classical, and the parameter-dependent principal symbol $\sym(p)(z,x,\zeta,\sigma)$ is defined on $\R_+\times [(T^*Z\times\R)\setminus 0]$ and satisfies
$$
\sym(p)(z,x,\varrho^{\ell_1}\zeta,\varrho^{\ell_2}\sigma) = \varrho^{\mu}\sym(p)(z,x,\zeta,\sigma), \quad \varrho > 0.
$$
We let
\begin{equation}\label{ApsiSym}
{}^b\sym_{\psi}(A):= \sym(p) \textup{ on } \R^+\times [(T^*Z\times\R)\setminus 0].
\end{equation}
As the quantization of $p(x,\sigma)$ leading to $A$ is based on the Mellin transform, ${}^b\sym_{\psi}(A)$ is a rescaled version of the usual principal symbol $\sym_{\psi}(A)$ of $A$ on $\R_+\times Z$ (see also Lemma~\ref{Compatibility}, but note that we allow $\ell_1 \neq \ell_2$ here). The relation is given by
$$
\sym_{\psi}(A)(z,x,\zeta,\xi) = {}^b\sym_{\psi}(A)(z,x,\zeta,x\xi),
$$
which is the same relation by rescaling as between the $b$-principal symbol and the ordinary principal symbol of a $b$-pseudodifferential operator on a manifold with boundary \cite{RBM2}.

The operator-valued symbol $a$ also has a parameter-dependent principal symbol $\sym(a)(z,x,\zeta,\sigma,\tau)$ of its own on $\R_+ \times [(T^*Z \times \R \times \overline{\R}_+) \setminus 0]$ that is anisotropic homogeneous in the sense that
$$
\sym(a)(z,x,\varrho^{\ell_1}\zeta,\varrho^{\ell_2}\sigma,\varrho^{\ell_3}\tau) = \varrho^{\mu}\sym(a)(z,x,\zeta,\sigma,\tau), \quad \varrho > 0.
$$
The expansion \eqref{pxstay} for $p(x,\sigma)$ shows that
\begin{equation}\label{prsymres}
\sym(p)(z,x,\zeta,\sigma) = \sym(a)(z,x,\zeta,\sigma,0).
\end{equation}
We define
\begin{equation}\label{AeSym}
\sym_e(A) = \sym(a) \textup{ on } \R_+ \times [(T^*Z \times \R \times \overline{\R}_+) \setminus 0].
\end{equation}
Relation \eqref{prsymres} shows that this \emph{extended principal symbol} of $A$ captures ${}^b\sym_{\psi}(A)$. We take a fiberwise view to explain this further: For every $(x,z) \in \R_+\times Z$, the extended principal symbol $\sym_e(A)$ is determined by its values on a closed anisotropic hemisphere
$$
|\zeta|^{2\ell_2\ell_3} + |\sigma|^{2\ell_1\ell_3} + |\tau|^{2\ell_1\ell_2} = 1, \quad \tau \geq 0,
$$
by homogeneity (an arbitrary Riemannian metric is used to measure the length of $\zeta \in T_z^*Z$). When restricting to $\tau = 0$ we obtain the anisotropic cosphere
$$
|\zeta|^{2\ell_2\ell_3} + |\sigma|^{2\ell_1\ell_3} = 1,
$$
and the extended principal symbol $\sym_e(A)$ restricted to that cosphere then matches ${}^b\sym_{\psi}(A)$. The relevant principal symbol for the calculus on $\R_+\times Z$ is the extension of ${}^b\sym_{\psi}(A)$ from the cosphere bundle to the hemisphere bundle over $\R_+\times Z$, given by $\sym_{e}(A)$.

As $\sym(a)$ determines $a(x,\sigma,\tau)$ modulo $C^{\infty}_B(\R_+,M^{\mu-1;\vec{\ell}}_{O,\textup{cl}}(Z;\overline{\R}_+))$ (that this is indeed the case in the holomorphic category of operator-valued symbols uses the kernel cut-off operator \eqref{HphiSpaces}), the extended principal symbol $\sym_e(A)$ determines $A$ modulo $\Psi^{\mu-1;\vec{\ell}}_{\textup{cl}}(\R_+\times Z)$.

\begin{lemma}
Let $a \in C_B^{\infty}(\R_+,M_{O}^{-\infty}(Z;\overline{\R}_+))$. Then $\opm(p) \in \Psi^{-\infty}_G(\R_+\times Z)$.
\end{lemma}
\begin{proof}
We need to verify the defining mapping properties. By assumption on $a(x,\sigma,\tau)$ we have
$$
x^{\beta} p(x,\sigma) = x^{\beta}a(x,\sigma,x) \in C^{\infty}_B(\R_+,M^{-\infty}_O(Z))
$$
for every $\beta \geq 0$, and consequently
$$
x^{\beta} \opm(p) : x^{\alpha}H^s_b(\R_+\times Z) \to x^{\alpha}H^{s'}_b(\R_+\times Z)
$$
for all $s,s' \in \R$ and all $\alpha \in \R$. Thus
$$
\opm(p) : x^{\alpha}H^s_b(\R_+\times Z) \to x^{\alpha'}H^{s'}_b(\R_+\times Z)
$$
for all $s,s',\alpha \in \R$ and all $\alpha' \leq \alpha$.

Similarly, for $\beta \geq 0$, we next want to prove that
$$
x^{\beta}\opm(p^{\star}) : x^{\alpha}H^s_b(\R_+\times Z) \to x^{\alpha}H^{s'}_b(\R_+\times Z)
$$
for all $s,s' \in \R$ and all $\alpha \in \R$, and thus
$$
[\opm(p)]^* = \opm(p^{\star}) : x^{\alpha}H^s_b(\R_+\times Z) \to x^{\alpha'}H^{s'}_b(\R_+\times Z)
$$
for all $s,s',\alpha \in \R$ and all $\alpha' \leq \alpha$. Here $p^{\star}(x,\sigma)$ is the Mellin symbol of the formal adjoint operator and is given by \eqref{MelAdjoint}. Now
$$
[x^{\beta}\opm(p^{\star})]^* = \opm(p)x^{\beta} = x^{\beta}\opm(p_{\beta})
$$
by \eqref{conjugate}, where $p_{\beta}(x,\sigma) = p(x,\sigma - i\beta) = a(x,\sigma - i\beta,x)$. By assumption on $a(x,\sigma,\tau)$ we have
$$
x^{\beta}p_{\beta}(x,\sigma) \in C^{\infty}_B(\R_+,M^{-\infty}_O(Z)),
$$
and from the standard Mellin calculus we get that
$$
[x^{\beta}\opm(p_{\beta})]^* = [x^{\beta}\opm(p^{\star})]^{**} = x^{\beta}\opm(p^{\star}) : x^{\alpha}H^s_b(\R_+\times Z) \to x^{\alpha}H^{s'}_b(\R_+\times Z)
$$
for all $s,s' \in \R$ and all $\alpha \in \R$ as desired.
\end{proof}

\begin{proposition}\label{AsymptoticComplete}
The calculus $\Psi^{\star;\vec{\ell}}(\R_+\times Z)$ is asymptotically complete. Suppose
$$
A_k = \opm(p_k) + G_k \in \Psi^{\mu_k;\vec{\ell}}(\R_+\times Z)
$$
with $\mu_k \geq \mu_{k+1} \to -\infty$ as $k \to \infty$, where
$$
p_k(x,\sigma) = a_k(x,\sigma,x), \quad a_k(x,\sigma,\tau) \in C^{\infty}_B(\R_+,M^{\mu_k;\vec{\ell}}_O(Z;\overline{\R}_+)),
$$
and $G_k \in \Psi^{-\infty}_G(\R_+\times Z)$. Then there exists
$$
a(x,\sigma,\tau) \in C^{\infty}_B(\R_+,M^{\mu_0;\vec{\ell}}_O(Z;\overline{\R}_+))
$$
with
$$
a(x,\sigma,\tau) \sim \sum\limits_{k=0}^{\infty}a_k(x,\sigma,\tau)
$$
in the sense that
$$
a(x,\sigma,\tau) - \sum\limits_{k=0}^{N-1}a_k(x,\sigma,\tau) \in C^{\infty}_B(\R_+,M^{\mu_N;\vec{\ell}}_O(Z;\overline{\R}_+))
$$
for every $N \in \N_0$, and with $A = \opm(p) \in \Psi^{\mu_0;\vec{\ell}}(\R_+\times Z)$, $p(x,\sigma) = a(x,\sigma,x)$, we then have
$$
A - \sum\limits_{k=0}^{N-1}A_k \in \Psi^{\mu_N;\vec{\ell}}(\R_+\times Z)
$$
for all $N \in \N_0$.
\end{proposition}
\begin{proof}
We need to show the existence of an operator-valued symbol $a(x,\sigma,\tau)$ with the prescribed asymptotic expansion. To do so, we first forget about the analyticity with respect to $\sigma \in \C$ and consider
$$
a_k(x,\sigma,\tau) \in C^{\infty}_B(\R_+,L^{\mu_k;\vec{\ell}}(Z;\R\times\overline{\R}_+)).
$$
The usual Borel argument in coordinates then shows that there exists $b(x,\sigma,\tau) \in C^{\infty}_B(\R_+,L^{\mu_0;\vec{\ell}}(Z;\R\times\overline{\R}_+))$ such that
$$
b(x,\sigma,\tau) - \sum\limits_{k=0}^{N-1}a_k(x,\sigma,\tau) \in C^{\infty}_B(\R_+,L^{\mu_N;\vec{\ell}}(Z;\R\times\overline{\R}_+))
$$
for every $N \in \N_0$. Now let
$$
a(x,\sigma,\tau) = [H(\phi)b](x,\sigma,\tau) \in C^{\infty}_B(\R_+,M^{\mu_0;\vec{\ell}}_O(Z;\overline{\R}_+))
$$
with the kernel cut-off operator $H(\phi)$ from \eqref{HphiSpaces}. Then
$$
b(x,\sigma,\tau) - [H(\phi)b](x,\sigma,\tau) \in C^{\infty}_B(\R_+,L^{-\infty}(Z;\R\times\overline{\R}_+)),
$$
and thus
\begin{align*}
a(x,\sigma,\tau) &- \sum\limits_{k=0}^{N-1}a_k(x,\sigma,\tau) \\
&\in C^{\infty}_B(\R_+,M^{\mu_0;\vec{\ell}}_O(Z;\overline{\R}_+))\cap C^{\infty}_B(\R_+,L^{\mu_N;\vec{\ell}}(Z;\R\times\overline{\R}_+)) \\
&= C^{\infty}_B(\R_+,M^{\mu_N;\vec{\ell}}_O(Z;\overline{\R}_+)).
\end{align*}
\end{proof}

\begin{proposition}
For every $A \in \Psi^{\mu;\vec{\ell}}_{\textup{(cl)}}(\R_+\times Z)$ and every $\beta \in \R$ we have
$x^{-\beta}Ax^{\beta} \in \Psi^{\mu;\vec{\ell}}_{\textup{(cl)}}(\R_+\times Z)$, and
$$
x^{-\beta}Ax^{\beta} - A \in \Psi^{\mu-\ell_2;\vec{\ell}}_{\textup{(cl)}}(\R_+\times Z).
$$
In particular, $\sym_e(A) = \sym_e(x^{-\beta}Ax^{\beta})$ in the classical case.
\end{proposition}
\begin{proof}
For the residual class this is proved in Lemma~\ref{ResGreen}, while by \eqref{conjugate} we have $x^{-\beta}\opm(p)x^{\beta} = \opm(p_{\beta})$ with
$$
p_{\beta}(x,\sigma) = p(x,\sigma-i\beta) = a(x,\sigma-i\beta,x).
$$
Here $a(x,\sigma,\tau)$ is an operator-valued symbol as stated in \eqref{asymbol}, and we have
$$
a(x,\sigma-i\beta,\tau) \sim \sum\limits_{j=0}^{\infty}\frac{(-i\beta)^j}{j!}\partial_{\sigma}^ja(x,\sigma,\tau)
$$
in that symbol class, which implies the claim.
\end{proof}

\begin{lemma}\label{Compatibility}
Let $\ell_1 = \ell_2$. Then
$$
\Psi^{\mu;\vec{\ell}}_{\textup{(cl)}}(\R_+\times Z) \subset L^{\mu;\ell_1}_{\textup{(cl)}}(\R_+\times Z).
$$
In the classical case, the homogeneous principal symbol $\sym_{\psi}(A)$ of $A$ as an element of $L^{\mu;\ell_1}_{\textup{cl}}(\R_+\times Z)$, and the $b$-principal symbol ${}^b\sym_{\psi}(A)$ defined in \eqref{ApsiSym}, are related by
$$
\sym_{\psi}(A)(z,x,\zeta,\xi) = {}^b\sym_{\psi}(A)(z,x,\zeta,x\xi).
$$
\end{lemma}
\begin{proof}
This is a consequence of the equivalence of phase functions based on either the Fourier or Mellin transform for the pseudodifferential quantization map. An explicit relationship is given by what is called Mellin quantization in Schulze's theory \cite{SchuNH}. Gil, Schulze, and Seiler \cite{GilSchulzeSeiler0,GilSchulzeSeiler} construct Mellin quantization maps
$$
Q : a(x,\xi) \mapsto p(x,\sigma) \qquad \tilde{Q} : p(x,\sigma) \mapsto a(x,\xi)
$$
on the level of (operator-valued) symbols that are explicitly given by oscillatory integral representations such that
$$
\opm(p) = \textup{op}(a(x,x\xi)) + R.
$$
The remainder term $R$ is smoothing and is also given explicitly. Here $\textup{op}$ refers to the usual Kohn-Nirenberg quantization.
This formula is to be read as either $p = Q(a)$ and then also $R = R(a)$ when passing from the Kohn-Nirenberg quantized operator to the Mellin operator, or as $a = \tilde{Q}(p)$ and $R = R(p)$ vice versa. The maps $Q$ and $\tilde{Q}$ are inverses of each other modulo regularizing (operator-valued) symbols, and there are explicit asymptotic expansions for $p = Q(a)$ in terms of $a$, and for $a = \tilde{Q}(p)$ in terms of $p$, which only involve $\xi$- or $\sigma$-derivatives of $a$ or $p$, respectively, powers of $\xi$ or $\sigma$, and certain universal combinatorial coefficients; in particular, according to these expansions, $p$ and $a$ differ by lower-order terms. The construction and proofs in an anisotropic setting can also be found in \cite{KrainerMellin}.
\end{proof}

\subsection{Compositions and formal adjoints}

\begin{lemma}\label{ImprovedMappingProps}
Let $p(x,\sigma) = a(x,\sigma,x)$ with $a(x,\sigma,\tau) \in C^{\infty}_B(\R_+,M^{\mu;\vec{\ell}}_O(Z;\overline{\R}_+))$ be arbitrary, and let $\mu \leq 0$. For $j \in \N_0$ with $\mu+\ell_3j \leq 0$ we have
$$
\opm(p),\; [\opm(p)]^* : x^{\alpha}H^{s;(\ell_1,\ell_2)}_b(\R_+\times Z) \to x^{\alpha-j}H^{s-(\mu+\ell_3 j);(\ell_1,\ell_2)}_b(\R_+\times Z)
$$
for all $s,\alpha \in \R$.
\end{lemma}
\begin{proof}
By assumption on $\mu$ and $j$, we have
$$
x^jp(x,\sigma) \in C^{\infty}_B(\R_+,M_O^{\mu+\ell_3 j;(\ell_1,\ell_2)}(Z)).
$$
and consequently $\opm(p)$ has the desired mapping properties. Now
$$
x^j[\opm(p)]^* = [\opm(p)x^j]^* = [x^j\opm(p_{j})]^*
$$
with $p_j(x,\sigma) = p(x,\sigma-ij) = a(x,\sigma-ij,x)$, and we have
$$
a(x,\sigma-ij,\tau) \in C^{\infty}_B(\R_+,M^{\mu;\vec{\ell}}_O(Z;\overline{\R}_+)).
$$
Consequently,
$$
x^jp_j(x,\sigma) \in C^{\infty}_B(\R_+,M_O^{\mu+\ell_3 j;(\ell_1,\ell_2)}(Z)),
$$
and thus the same is true for the operator-valued symbol of the formal adjoint, which implies that
$$
x^j[\opm(p)]^*= [x^j\opm(p_j)]^* : x^{\alpha}H^{s;(\ell_1,\ell_2)}_b(\R_+\times Z) \to x^{\alpha}H^{s-(\mu+\ell_3 j);(\ell_1,\ell_2)}_b(\R_+\times Z).
$$
\end{proof}

\begin{theorem}\label{AdjointsCalculus}
Let $A = \opm(p) + G \in \Psi^{\mu;\vec{\ell}}_{\textup{(cl)}}(\R_+\times Z)$. Then the formal adjoint of $A$ with respect to the $L^2_b$-inner product belongs to $\Psi^{\mu;\vec{\ell}}_{\textup{(cl)}}(\R_+\times Z)$. More precisely, suppose that $p(x,\sigma) = a(x,\sigma,x)$ with $a(x,\sigma,\tau) \in C^{\infty}_B(\R_+,M^{\mu;\vec{\ell}}_O(Z;\overline{\R}_+))$. Let then $a^{\star}(x,\sigma,\tau)  \in C^{\infty}_B(\R_+,M^{\mu;\vec{\ell}}_O(Z;\overline{\R}_+))$ be arbitrary with asymptotic expansion
\begin{equation}\label{AdjointAsymptotics}
a^{\star}(x,\sigma,\tau) \sim \sum\limits_{k=0}^{\infty}\frac{1}{k!}[(-xD_x - \tau D_{\tau})^k\partial_{\sigma}^k a](x,\overline{\sigma},\tau)^*
\end{equation}
within this operator-valued symbol class, where the $*$ operator that appears in the asymptotic terms in \eqref{AdjointAsymptotics} is the formal adjoint with respect to the $L^2$-inner product on $Z$. Then there exists $G' \in \Psi^{-\infty}_G(\R_+\times Z)$ such that
$$
A^* = \opm(a^{\star}(x,\sigma,x) ) + G' \in \Psi^{\mu;\vec{\ell}}_{\textup{(cl)}}(\R_+\times Z).
$$
In the classical case we find that $\sym_e(A^*) = \overline{\sym_e(A)}$ (or fiberwise adjoints when the operators are acting in sections of the pull-back to $\R_+\times Z$ of a Hermitian vector bundle on $Z$).
\end{theorem}
\begin{proof}
Observe that an operator-valued symbol $a^{\star}(x,\sigma,\tau)$ with the prescribed asymptotic expansion \eqref{AdjointAsymptotics} exists by Proposition~\ref{AsymptoticComplete}.

We first consider $a(x,\sigma,\tau) \in C^{\infty}_B(\R_+,M^{-\infty}_O(Z;\overline{\R}_+))$. In particular, the formulas \eqref{MelAdjoint} and \eqref{MelAdjointRem} for the formal adjoint from the standard Mellin calculus are applicable to $\opm(p)$. We get that $[\opm(p)]^* = \opm(p^{\star})$ with the symbol $p^{\star}(x,\sigma)$ explicitly given by \eqref{MelAdjoint}, and for every $N \in \N$ we have
\begin{align*}
p^{\star}(x,\sigma) &= \sum\limits_{k=0}^{N-1}\frac{1}{k!}[(-xD_x)^k\partial_{\sigma}^k p](x,\overline{\sigma})^*+ r_N(x,\sigma) \\
&= \sum\limits_{k=0}^{N-1}\frac{1}{k!}[(-xD_x - \tau D_{\tau})^k\partial_{\sigma}^k a](x,\overline{\sigma},x)^* + r_N(x,\sigma) \\
&= p_N(x,\sigma) + r_N(x,\sigma),
\end{align*}
where
$$
r_N(x,\sigma) = \frac{1}{2\pi}\!\int_0^1\frac{(1-\theta)^{N-1}}{(N-1)!}\iint y^{-i\eta}a_N(xy,\overline{\sigma}+\theta\eta,xy)^*\,\frac{dy}{y}d\eta d\theta.
$$
Here
$$
a_N(x,\sigma,\tau) = (-xD_x - \tau D_{\tau})^{N}\partial_{\sigma}^Na(x,\sigma,\tau).
$$
For any $j \in \N_0$ we  now consider $x^jr_N(x,\sigma)$. In view of the analyticity of the symbols we can shift integration to the complex $\eta$-plane in the formula for $r_N(x,\sigma)$ and get
\begin{align*}
x^jr_N(x,\sigma) &= \frac{1}{2\pi}\!\int_0^1\frac{(1-\theta)^{N-1}}{(N-1)!}\iint y^{-i\eta}x^ja_N(xy,\overline{\sigma}+\theta\eta,xy)^*\,\frac{dy}{y}d\eta d\theta \\
&= \frac{1}{2\pi}\!\int_0^1\frac{(1-\theta)^{N-1}}{(N-1)!}\iint y^{-i\eta}(xy)^ja_N(xy,\overline{\sigma}+\theta(\eta-ij),xy)^*\,\frac{dy}{y}d\eta d\theta \\
&= \frac{1}{2\pi}\!\int_0^1\frac{(1-\theta)^{N-1}}{(N-1)!}\iint y^{-i\eta}[\tau^ja_N](xy,\overline{\sigma}+\theta(\eta-ij),xy)^*\,\frac{dy}{y}d\eta d\theta.
\end{align*}
The map $C^{\infty}_B(\R_+,M^{\mu';\vec{\ell}}_O(Z;\overline{\R}_+)) \to C^{\infty}_B(\R_+,M^{\mu'-N\ell_2 + j\ell_3;\vec{\ell}}_O(Z;\overline{\R}_+))$ given by
$$
a(x,\sigma,\tau) \mapsto \tau^ja_N(x,\sigma,\tau)
$$
is continuous for all $\mu' \in \R$ and $j,N \in \N_0$. The oscillatory integral formula for $x^jr_N(x,\sigma)$ implies that, for any given $\mu' \in \R$, $j_0 \in \N_0$, and $\mu_0 < 0$, there exists $N_0 \in \N$ sufficiently large such that, for $N \geq N_0$ and $0 \leq j \leq j_0$, the map
$$
C^{\infty}_B(\R_+,M^{\mu';\vec{\ell}}_O(Z;\overline{\R}_+)) \ni a(x,\sigma,\tau) \mapsto x^jr_N(x,\sigma) \in C^{\infty}_B(\R_+,M^{\mu_0;(\ell_1,\ell_2)}_O(Z))
$$
is continuous. Consequently, for $a(x,\sigma,\tau) \in C^{\infty}_B(\R_+,M^{\mu';\vec{\ell}}_O(Z;\overline{\R}_+))$, we get
$$
\opm(r_N) : x^{\alpha}H^{s;(\ell_1,\ell_2)}_b(\R_+\times Z) \to x^{\alpha-j}H^{s-\mu_0;(\ell_1,\ell_2)}_b(\R_+\times Z)
$$
for all $\alpha,s \in \R$. The same reasoning shows that the map
$$
C^{\infty}_B(\R_+,M^{\mu';\vec{\ell}}_O(Z;\overline{\R}_+)) \ni a(x,\sigma,\tau) \mapsto x^jr_N(x,\sigma-ij) \in C^{\infty}_B(\R_+,M^{\mu_0;(\ell_1,\ell_2)}_O(Z))
$$
is continuous for $N \geq N_0$ and $0 \leq j \leq j_0$, and consequently
\begin{align*}
[\opm(r_N)]^* = x^{-j}&[\opm(x^jr_N(x,\sigma-ij))]^* : \\
&x^{\alpha}H^{s;(\ell_1,\ell_2)}_b(\R_+\times Z) \to x^{\alpha-j}H^{s-\mu_0;(\ell_1,\ell_2)}_b(\R_+\times Z)
\end{align*}
for all $\alpha,s \in \R$ as well.

Now let $a \in C^{\infty}_B(\R_+,M^{\mu;\vec{\ell}}_O(Z;\overline{\R}_+))$. By Lemma~\ref{AnalyticSmoothingApprox} there exists a sequence $a_{\nu} \in C^{\infty}_B(\R_+,M^{-\infty}_O(Z;\overline{\R}_+))$ such that $a_{\nu} \to a$ in $C^{\infty}_B(\R_+,M^{\mu';\vec{\ell}}_O(Z;\overline{\R}_+))$ as $\nu \to \infty$ for $\mu' > \mu$. Let $u,v \in C_c^{\infty}(\R_+\times Z)$ be arbitrary. Then
\begin{gather*}
\langle \opm(a(x,\sigma,x))u,v \rangle_{L^2_b} \longleftarrow \langle \opm(a_{\nu}(x,\sigma,x))u,v \rangle_{L^2_b} \\
= \langle u,[\opm(a_{\nu}(x,\sigma,x))]^*v \rangle_{L^2_b} 
= \langle u, (\opm(p_{N,\nu}) + \opm(r_{N,\nu}))v \rangle_{L^2_b} \\
\longrightarrow \langle u, (\opm(p_{N}) + \opm(r_{N}))v \rangle_{L^2_b}
\end{gather*}
for $N \in \N$ large enough, and $\opm(r_N)$ and $[\opm(r_N)]^*$ have the mapping properties in the anisotropic $b$-Sobolev spaces stated above. In particular, $[\opm(a(x,\sigma,x))]^* = \opm(p_N) + \opm(r_N)$.

Consequently, if $a^{\star}(x,\sigma,\tau)$ has the asymptotic expansion \eqref{AdjointAsymptotics}, then $G_0 = [\opm(a(x,\sigma,x)]^* - \opm(a^{\star}(x,\sigma,x))$ satisfies
$$
G_0,\,G_0^* : x^{\alpha}H^{s;(\ell_1,\ell_2)}_b(\R_+\times Z) \to x^{\alpha'}H^{s';(\ell_1,\ell_2)}_b(\R_+\times Z)
$$
for all $\alpha,s,s' \in \R$, and all $\alpha' \leq \alpha$ by Lemma~\ref{ImprovedMappingProps} and the mapping properties of $\opm(r_N)$ (as $N \to \infty$), which shows that $G_0 \in \Psi^{-\infty}_G(\R_+\times Z)$. The theorem is proved.
\end{proof}

\begin{theorem}\label{CompositionCalculus}
Let $A_j = \opm(a_j(x,\sigma,x)) + G_j \in \Psi^{\mu_j;\vec{\ell}}_{\textup{(cl)}}(\R_+\times Z)$, $j = 1,2$, with
$$
a_j(x,\sigma,\tau) \in C^{\infty}_B(\R_+,M^{\mu_j;\vec{\ell}}_{O,\textup{(cl)}}(Z;\overline{\R}_+)).
$$
Then the composition $A_1 A_2 \in \Psi^{\mu_1+\mu_2;\vec{\ell}}_{\textup{(cl)}}(\R_+\times Z)$. More precisely, if $a_1{\#}a_2 \in C^{\infty}_B(\R_+,M^{\mu_1+\mu_2;\vec{\ell}}_{O,\textup{(cl)}}(Z;\overline{\R}_+))$ has the asymptotic expansion
\begin{equation}\label{LeibnizAsymptotics}
(a_1{\#}a_2)(x,\sigma,\tau) \sim \sum\limits_{k=0}^{\infty}\frac{1}{k!}[\partial_{\sigma}^ka_1](x,\sigma,\tau)[(xD_x + \tau D_{\tau})^ka_2](x,\sigma,\tau),
\end{equation}
then there exists $G \in \Psi^{-\infty}_G(\R_+\times Z)$ such that
$$
A_1 A_2 = \opm((a_1{\#}a_2)(x,\sigma,x)) + G.
$$
In the classical case we have $\sym_e(A_1 A_2) = \sym_e(A_1)\sym_e(A_2)$.
\end{theorem}
\begin{proof}
The composition $G_1G_2 \in \Psi^{-\infty}_G(\R_+\times Z)$ in view of the defining mapping properties of the residual class. We next prove that the compositions
$$
\opm(a_1(x,\sigma,x))G_2 \textup{ and } G_1\opm(a_2(x,\sigma,x))
$$
are residual as well. Let $\omega \in C_c^{\infty}(\overline{\R}_+)$ with $\omega \equiv 1$ near $x=0$, and write
$$
a_j(x,\sigma,\tau) = \omega(x) a_j(x,\sigma,\tau) + (1-\omega)(x) a_j(x,\sigma,\tau) = a_{j,0}(x,\sigma,\tau) + a_{j,\infty}(x,\sigma,\tau).
$$
For $K > 0$ large enough we have
$$
a_{j,0}(x,\sigma,x),\, x^{-K}a_{j,\infty}(x,\sigma,x) \in C^{\infty}_B(\R_+,M^{\mu_j;(\ell_1,\ell_2)}_O(Z)),
$$
and consequently
\begin{gather*}
\opm(a_{j,0}(x,\sigma,x)) : x^{\alpha}H^{s;(\ell_1,\ell_2)}_b(\R_+\times Z) \to x^{\alpha}H^{s-\mu_j;(\ell_1,\ell_2)}_b(\R_+\times Z), \\
\opm(a_{j,\infty}(x,\sigma,x)) : x^{\alpha}H^{s;(\ell_1,\ell_2)}_b(\R_+\times Z) \to x^{\alpha+K}H^{s-\mu_j;(\ell_1,\ell_2)}_b(\R_+\times Z)
\end{gather*}
for all $s,\alpha \in \R$. By the mapping properties of $G_1$ we obtain
$$
G_1 \opm(a_{2,0}(x,\sigma,x)) : x^{\alpha}H^{s;(\ell_1,\ell_2)}_b(\R_+\times Z) \to x^{\alpha'}H^{s';(\ell_1,\ell_2)}_b(\R_+\times Z)
$$
for all $s,s' \in \R$ and all $\alpha' \leq \alpha$, and
$$
G_1 \opm(a_{2,\infty}(x,\sigma,x)) : x^{\alpha}H^{s;(\ell_1,\ell_2)}_b(\R_+\times Z) \to x^{\alpha'}H^{s';(\ell_1,\ell_2)}_b(\R_+\times Z)
$$
for all $s,s' \in \R$ and all $\alpha' \leq \alpha + K$, and so also for all $\alpha' \leq \alpha$, which shows that
$$
G_1 \opm(a_{2}(x,\sigma,x)) : x^{\alpha}H^{s;(\ell_1,\ell_2)}_b(\R_+\times Z) \to x^{\alpha'}H^{s';(\ell_1,\ell_2)}_b(\R_+\times Z).
$$
for all $s,s' \in \R$ and all $\alpha' \leq \alpha$. Because
$$
G_2 : x^{\alpha}H^{s;(\ell_1,\ell_2)}_b(\R_+\times Z) \to x^{\alpha'}H^{s'+\mu_1;(\ell_1,\ell_2)}_b(\R_+\times Z)
$$
we get
$$
\opm(a_{1,0}(x,\sigma,x))G_2 : x^{\alpha}H^{s;(\ell_1,\ell_2)}_b(\R_+\times Z) \to x^{\alpha'}H^{s';(\ell_1,\ell_2)}_b(\R_+\times Z)
$$
for all $s,s'  \in \R$ and $\alpha' \leq \alpha$, and similarly
$$
G_2 : x^{\alpha}H^{s;(\ell_1,\ell_2)}_b(\R_+\times Z) \to x^{\alpha'-K}H^{s'+\mu_1;(\ell_1,\ell_2)}_b(\R_+\times Z)
$$
and
$$
\opm(a_{1,\infty}(x,\sigma,x))G_2 : x^{\alpha}H^{s;(\ell_1,\ell_2)}_b(\R_+\times Z) \to x^{\alpha'}H^{s';(\ell_1,\ell_2)}_b(\R_+\times Z),
$$
and so
$$
\opm(a_1(x,\sigma,x))G_2 : x^{\alpha}H^{s;(\ell_1,\ell_2)}_b(\R_+\times Z) \to x^{\alpha'}H^{s';(\ell_1,\ell_2)}_b(\R_+\times Z)
$$
for all $s,s' \in \R$ and all $\alpha' \leq \alpha$. For the formal adjoints we have
\begin{align*}
[G_1\opm(a_2(x,\sigma,x))]^* &= [\opm(a_2(x,\sigma,x))]^*G_1^*, \\
[\opm(a_1(x,\sigma,x))G_2]^* &= G_2^*[\opm(a_1(x,\sigma,x))]^*,
\end{align*}
and because $[\opm(a_j)(x,\sigma,x)]^* \in \Psi^{\mu_j;\vec{\ell}}_{\textup{(cl)}}(\R_+\times Z)$ by Theorem~\ref{AdjointsCalculus} we obtain with the above that both $\opm(a_1(x,\sigma,x))G_2$ and $G_1\opm(a_2(x,\sigma,x))$ are residual.

It remains to consider the composition $\opm(a_1(x,\sigma,x))\opm(a_2(x,\sigma,x))$. We first consider
$$
a_j(x,\sigma,\tau)  \in C^{\infty}_B(\R_+,M^{-\infty}_{O}(Z;\overline{\R}_+)).
$$
We can then compose $\opm(a_1(x,\sigma,x))\opm(a_2(x,\sigma,x)) = \opm(c(x,\sigma))$ by the standard Mellin pseudodifferential calculus, and the formulas \eqref{MelLeibniz} and \eqref{MelLeibnizRem} apply. For every $N \in \N$
\begin{align*}
c(x,\sigma) &= \sum\limits_{k=0}^{N-1}\frac{1}{k!}[\partial_{\sigma}^ka_1](x,\sigma,x)[(xD_x + \tau D_{\tau})^ka_2](x,\sigma,x) + r_N(x,\sigma) \\
&= c_N(x,\sigma) + r_{N}(x,\sigma).
\end{align*}
We further decompose
$$
r_N(x,\sigma) = r_{N,0}(x,\sigma) + r_{N,\infty}(x,\sigma)
$$
with
\begin{align*}
r_{N,0}(x,\sigma) = \frac{1}{2\pi}\! &\int_0^1\!\frac{(1-\theta)^{N-1}}{(N-1)!} \iint y^{-i\eta} \cdot \\
&[\partial_{\sigma}^Na_1](x,\sigma+\theta\eta,x) [(xD_{x} + \tau D_{\tau})^Na_{2,0}](xy,\sigma,xy)\,\frac{dy}{y}d\eta d\theta, \\
r_{N,\infty}(x,\sigma) = \frac{1}{2\pi}\! &\int_0^1\!\frac{(1-\theta)^{N-1}}{(N-1)!} \iint y^{-i\eta} \cdot \\
&[\partial_{\sigma}^Na_1](x,\sigma+\theta\eta,x) [(xD_{x} + \tau D_{\tau})^Na_{2,\infty}](xy,\sigma,xy)\,\frac{dy}{y}d\eta d\theta,
\end{align*}
where $a_{2,0} = \omega a_2$ and $a_{2,\infty} = (1-\omega)a_2$ as above.

\medskip

\noindent
Let $\mu_1',\mu_2' \in \R$, $j_0 \in \N_0$, and $\mu_0 < 0$ be arbitrary.

We first analyze $r_{N,0}(x,\sigma)$: For $0 \leq j \leq j_0$ write
\begin{align*}
x^jr_{N,0}(x,\sigma) = \frac{1}{2\pi}\! &\int_0^1\!\frac{(1-\theta)^{N-1}}{(N-1)!} \iint y^{-i\eta} \cdot \\
&[\tau^j\partial_{\sigma}^Na_1](x,\sigma+\theta\eta,x) [(xD_{x} + \tau D_{\tau})^Na_{2,0}](xy,\sigma,xy)\,\frac{dy}{y}d\eta d\theta.
\end{align*}
Because
$$
C^{\infty}_B(\R_+,M^{\mu_2';\vec{\ell}}_{O}(Z;\overline{\R}_+)) \ni a_2(x,\sigma,\tau) \mapsto a_{2,0}(x,\sigma,x) \in C^{\infty}_B(\R_+,M^{\mu_2';(\ell_1,\ell_2)}_{O}(Z))
$$
is continuous we see from the formula that there exists $N_0 \in \N$ such that for all $N \geq N_0$ and all $0 \leq j \leq j_0$ the maps
$$
(a_1,a_2) \mapsto x^jr_{N,0}(x,\sigma) \textup{ and } (a_1,a_2) \mapsto x^jr_{N,0}(x,\sigma-ij)
$$
are continuous in
$$
C^{\infty}_B(\R_+,M^{\mu_1';\vec{\ell}}_{O}(Z;\overline{\R}_+)) \times C^{\infty}_B(\R_+,M^{\mu_2';\vec{\ell}}_{O}(Z;\overline{\R}_+)) \to C^{\infty}_B(\R_+,M^{\mu_0;(\ell_1,\ell_2)}_{O}(Z)),
$$
and so
\begin{gather*}
\opm(r_{N,0}),\, [\opm(r_{N,0})]^* = x^{-j} [\opm(x^jr_{N,0}(x,\sigma-ij))]^* : \\
x^{\alpha}H^{s;(\ell_1,\ell_2)}_b(\R_+\times Z) \to x^{\alpha-j}H^{s-\mu_0;(\ell_1,\ell_2)}_b(\R_+\times Z)
\end{gather*}
for all $\alpha,s \in \R$.

We next want to show the same mapping properties for $\opm(r_{N,\infty})$ for $N$ sufficiently large. Pick $K > 0$ large enough such that
$$
C^{\infty}_B(\R_+,M^{\mu_2';\vec{\ell}}_{O}(Z;\overline{\R}_+)) \!\ni\! a_2(x,\sigma,\tau) \mapsto x^{-K}a_{2,\infty}(x,\sigma,x) \!\in\! C^{\infty}_B(\R_+,M^{\mu_2';(\ell_1,\ell_2)}_{O}(Z))
$$
is continuous. We use the analyticity of the operator-valued symbols in the oscillatory integral formula for $r_{N,\infty}(x,\sigma)$ to shift the integration contour and write for every $0 \leq j \leq j_0$
\begin{gather*}
x^jr_{N,\infty}(x,\sigma) = \frac{1}{2\pi}\! \int_0^1\!\frac{(1-\theta)^{N-1}}{(N-1)!} \iint y^{-i\eta} \cdot \\
[\tau^{j+K}\partial_{\sigma}^Na_1](x,\sigma+\theta(\eta-iK),x)[x^{-K}(xD_{x}+\tau D_{\tau})^Na_{2,\infty}](xy,\sigma,xy)\,\frac{dy}{y}d\eta d\theta.
\end{gather*}
This formula shows that we can find $N_0 \in \N$ such that for $N \geq N_0$ and $0 \leq j \leq j_0$ the maps
$$
(a_1,a_2) \mapsto x^jr_{N,\infty}(x,\sigma) \textup{ and } (a_1,a_2) \mapsto x^jr_{N,\infty}(x,\sigma-ij)
$$
are continuous in
$$
C^{\infty}_B(\R_+,M^{\mu_1';\vec{\ell}}_{O}(Z;\overline{\R}_+)) \times C^{\infty}_B(\R_+,M^{\mu_2';\vec{\ell}}_{O}(Z;\overline{\R}_+)) \to C^{\infty}_B(\R_+,M^{\mu_0;(\ell_1,\ell_2)}_{O}(Z)),
$$
so $\opm(r_{N,\infty})$ and $[\opm(r_{N,\infty})]^*$ have the desired mapping properties. We conclude that
$$
\opm(r_{N}),\, [\opm(r_{N})]^* : x^{\alpha}H^{s;(\ell_1,\ell_2)}_b(\R_+\times Z) \to x^{\alpha-j}H^{s-\mu_0;(\ell_1,\ell_2)}_b(\R_+\times Z)
$$
for all $\alpha,s \in \R$.

Now let $a_j(x,\sigma,\tau) \in C^{\infty}_B(\R_+,M^{\mu_j;\vec{\ell}}_{O,\textup{(cl)}}(Z;\overline{\R}_+))$. By Lemma~\ref{AnalyticSmoothingApprox} there exist sequences $a_{j,\nu} \in C^{\infty}_B(\R_+,M^{-\infty}_{O}(Z;\overline{\R}_+))$ such that $a_{j,\nu}(x,\sigma,\tau) \to a_j(x,\sigma,\tau)$ as $\nu \to \infty$ in $C^{\infty}_B(\R_+,M^{\mu_j;\vec{\ell}}_{O,\textup{(cl)}}(Z;\overline{\R}_+))$ for $\mu_j' > \mu_j$.

For every $u \in \S_0(\R_+\times Z)$ and $N \in \N$ large enough
\begin{gather*}
\opm(a_1(x,\sigma,x))\opm(a_2(x,\sigma,x))u \longleftarrow \opm(a_{1,\nu}(x,\sigma,x))\opm(a_{2,\nu}(x,\sigma,x))u \\
= \opm(c_{N,\nu})u + \opm(r_{N,\nu})u \longrightarrow \opm(c_N)u + \opm(r_N)u.
\end{gather*}
So $\opm(a_1(x,\sigma,x))\opm(a_2(x,\sigma,x)) = \opm(c_N) + \opm(r_N)$ for $N$ large enough, and $\opm(r_N)$ has the mapping properties previously shown. Consequently, if the operator-valued symbol $a_1{\#}a_2 \in C^{\infty}_B(\R_+,M^{\mu_1+\mu_2;\vec{\ell}}_{O,\textup{(cl)}}(Z;\overline{\R}_+))$ has the asymptotic expansion \eqref{LeibnizAsymptotics}, then
$$
G_0 = \opm(a_1(x,\sigma,x))\opm(a_2(x,\sigma,x)) - \opm((a_1{\#}a_2)(x,\sigma,x))
$$
satisfies
$$
G_0,\,G_0^* : x^{\alpha}H^{s;(\ell_1,\ell_2)}_b(\R_+\times Z) \to x^{\alpha'}H^{s';(\ell_1,\ell_2)}_b(\R_+\times Z)
$$
for all $\alpha,s,s' \in \R$ and all $\alpha' \leq \alpha$ by Lemma~\ref{ImprovedMappingProps} and the mapping properties of $\opm(r_N)$ (as $N \to \infty$). The theorem is proved.
\end{proof}

\subsection{Ellipticity and parametrices on {\boldmath ${\mathbb R}_+\times Z$}}

\begin{definition}
The operator $A = \opm(a(x,\sigma,x)) + G \in \Psi^{\mu;\vec{\ell}}_{\textup{cl}}(\R_+\times Z)$ with $a(x,\sigma,\tau) \in C^{\infty}_B(\R_+,M^{\mu;\vec{\ell}}_{O,\textup{cl}}(Z;\overline{\R}_+))$ is fully elliptic if $\sym_e(A)(z,x,\zeta,\sigma,\tau)$ is invertible for all $(z,\zeta,\sigma,\tau) \in [T^*Z \times \R \times \overline{\R}_+] \setminus 0$ and all $x > 0$, and the inverse
$$
|\sym^{-1}_e(A)(z,x,\zeta,\sigma,\tau)|
$$
is bounded as $x$ varies in $\R_+$ and $(z,\zeta,\sigma,\tau) \in T^*Z \times \R \times \overline{\R}_+$ with $|\zeta|_{g_Z(z)}^{2\ell_2\ell_3} + |\sigma|^{2\ell_1\ell_3} + |\tau|^{2\ell_1\ell_2} = 1$, $z \in Z$, where $g_Z$ is a Riemannian metric on $Z$.
\end{definition}

\begin{theorem}\label{ParametrixModel}
Let $A \in \Psi^{\mu;\vec{\ell}}_{\textup{cl}}(\R_+\times Z)$ be fully elliptic. The there exists $P \in \Psi^{-\mu;\vec{\ell}}_{\textup{cl}}(\R_+\times Z)$ such that $AP - I,\, PA  - I \in \Psi^{-\infty}_G(\R_+\times Z)$.
\end{theorem}
\begin{proof}
By assumption
$$
a(x,\sigma,\tau) \in C^{\infty}_B(\R_+,L^{\mu;\vec{\ell}}_{\textup{cl}}(Z;\R\times\overline{\R}_+))
$$
is parameter-dependent elliptic, uniformly with respect to $x \in \R_+$. There exists a parameter-dependent parametrix
$$
b(x,\sigma,\tau) \in C^{\infty}_B(\R_+,L^{-\mu;\vec{\ell}}_{\textup{cl}}(Z;\R\times\overline{\R}_+))
$$
such that $ab - 1,\, ba - 1 \in C^{\infty}_B(\R_+,L^{-\infty}(Z;\R\times\overline{\R}_+))$. Replacing $b$ by $H(\phi)b$ with the kernel cut-off operator \eqref{HphiSpaces} gives $b \in C^{\infty}_B(\R_+,M^{-\mu;\vec{\ell}}_{O,\textup{cl}}(Z;\overline{\R}_+))$ such that
$$
ab - 1,\,ba - 1 \in C^{\infty}_B(\R_+,M^{-\infty}_O(Z;\overline{\R}_+)).
$$
Now $B = \opm(b(x,\sigma,x)) \in \Psi^{-\mu;\vec{\ell}}_{\textup{cl}}(\R_+\times Z)$, and $AB - I,\, BA - I \in \Psi^{-1;\vec{\ell}}_{\textup{cl}}(\R_+\times Z)$ by Theorem~\ref{CompositionCalculus}. By the composition theorem and asymptotic completeness (Proposition~\ref{AsymptoticComplete}) the formal Neumann series argument is applicable which yields the desired parametrix $P \in \Psi^{-\mu;\vec{\ell}}_{\textup{cl}}(\R_+\times Z)$.
\end{proof}


\section{The calculus on $M$}

\noindent
Let $M$ be an asymptotically conic manifold as described in Section~\ref{ManEnds}. On $M$ we fix a positive density that on the noncompact end for large $x \gg 1$ is given by $\frac{dx}{x}dz$, where $dz$ the positive density associated with an arbitrary (but henceforth fixed) Riemannian metric $g_Z$ on the cross-section $Z$. Let $L^2_b(M)$ be the $L^2$-space associated with that density, which is the base space for our analysis on $M$. While we are going to discuss the details of the calculus for scalar-valued functions here, we point out that the constructions equally work for operators acting in sections of vector bundles over $M$. On the noncompact end, a Hermitian vector bundle $E$ is isometric to the pull-back of a Hermitian vector bundle on $Z$ (where $Z$ is identified as $x=1$)
$$
E\big|_{[1,\infty)\times Z} \cong \pi^* E\big|_{Z}, \quad \pi : [1,\infty)\times Z \to Z.
$$
Such an isometry can be obtained by parallel transport with respect to a compatible connection on $E$ along the integral curves of the vector field $\partial_x$. The calculus on the noncompact end is then defined as acting in sections of pull-back bundles from $Z$, and is based on quantizations of symbols, as discussed in Section~\ref{PseudoCalculusExit}, but now taking values in operators on $Z$ that act in sections of vector bundles over $Z$. There is invariance with the requirement that the transition map, as well as its inverse, which relate different isomorphisms $E\big|_{[1,\infty)\times Z} \cong \pi^* E\big|_{Z}$, are operators in the calculus of order $0$.

\medskip

\noindent
Throughout this section the anisotropy vector $\vec{\ell} = (\ell_1,\ell_1,\ell_3) \in \N^3$ is fixed.

\subsection{Residual operators}

For $s \in \R$ let $H^s_b(M)$ denote the space of all $u \in H^{s}_{\textup{loc}}(M)$ such that for some $\phi \in C_c^{\infty}(M)$ with $\phi \equiv 1$ when $x \leq \frac{1}{2}$ and $\phi \equiv 0$ for $x \geq \frac{3}{4}$, we have
$$
(1 - \phi)u \in H^s_b(\R_+\times Z).
$$
Recall that $x$ is globally defined on $M$, see Section~\ref{ManEnds}. Observe that we use this $b$-Sobolev space to control behavior of functions as $x \to \infty$ here, while typically $b$-Sobolev spaces are utilized to control behavior as $x \to 0$.

This $b$-Sobolev space $H^s_b(M)$ can be given a Hilbert space structure; it is realized as a quotient of the direct sum of $H^s(2M_0) \times H^s_b(\R_+\times Z)$, where $2M_0 = M_0 \sqcup_{Z} M_0$ is the double of the compact manifold $M_0$ with boundary $Z$, see Section~\ref{ManEnds}, by identifying tuples modulo the kernel of the map $(v,w) \mapsto \phi v + (1-\phi)w$. We have $H^0_b(M) = L^2_b(M)$, continuous embeddings $x^{\alpha}H^s_b(M) \hookrightarrow x^{\alpha'}H^{s'}_b(M)$ for $s \geq s'$ and $\alpha \leq \alpha'$ (compact if both inequalities are strict), and duality $[x^{\alpha}H^s_b(M)]' \cong x^{-\alpha}H^{-s}_b(M)$ with respect to the extended $L^2_b$-inner product. Moreover,
$$
\S(M) = \bigcap\limits_{s,\alpha \in \N_0}x^{-\alpha}H^s_b(M), \quad \S'(M) = \bigcup\limits_{s,\alpha \in \N_0}x^{\alpha}H^{-s}_b(M).
$$

\medskip

\noindent
The class $\Psi^{-\infty}_G(M)$ of residual operators on $M$ is defined to consist of all operators
$$
G : \S(M) \to \S(M)
$$
that are integral operators with kernels in $\S(M\times M)$ with respect to the $b$-density on $M$. In terms of the $b$-Sobolev spaces, $G \in \Psi^{-\infty}_G(M)$ if and only if
$$
G : x^{\alpha}H^s_b(M) \to x^{\alpha'}H^{s'}_b(M)
$$
is continuous for all $s,s',\alpha,\alpha' \in \R$.

Observe that with $G$ also $G^* \in \Psi^{-\infty}_G(M)$. Moreover, $x^{\beta}G,\, G x^{\beta} \in \Psi^{-\infty}_G(M)$ for all $\beta \in \R$.

\subsection{Definition and properties of the calculus on {\boldmath $M$}}

The class $\Psi^{\mu;\vec{\ell}}_{\textup{(cl)}}(M)$ consists of all continuous operators
$$
A : \S(M) \to \S(M)
$$
with the following properties:
\begin{itemize}
\item For all $\phi,\psi \in C_c^{\infty}(M)$ with $\textup{supp}(\phi)\cap\textup{supp}(1-\psi) = \emptyset$ we have
$$
\phi A (1 - \psi),\, (1-\psi) A \phi \in \Psi^{-\infty}_G(M).
$$
\item For all $\phi,\psi \in C_c^{\infty}(M)$, the operator $\phi A \psi \in L^{\mu;\ell_1}_{\textup{(cl)}}(M)$, see Section~\ref{step1lcalc}.
\item There exists an $R \gg 1$ depending on $A$ such that, for all $\phi,\psi \in C_c^{\infty}(M)$ with both $\textup{supp}(1-\phi)$ and $\textup{supp}(1-\psi)$ contained in $x > R$, we have
$$
(1-\phi)A(1-\psi) \in \Psi^{\mu;\vec{\ell}}_{\textup{(cl)}}(\R_+\times Z).
$$ 
\end{itemize}
By Lemma~\ref{Compatibility} we have $\Psi^{\mu;\vec{\ell}}_{\textup{(cl)}}(M) \subset L^{\mu;\ell_1}_{\textup{(cl)}}(M)$. In particular, every $A \in \Psi^{\mu;\vec{\ell}}_{\textup{cl}}(M)$ has a homogeneous principal symbol $\sym_{\psi}(A)$ on $T^*M \setminus 0$, where
$$
\sym_{\psi}(A)(y,\varrho^{\ell_1}\eta) = \varrho^{\mu}\sym_{\psi}(A)(y,\eta), \quad \varrho > 0.
$$
Moreover, for $R \gg 1$ large enough, we can restrict $A$ to an operator
$$
C_c^{\infty}(\{x > R\}) \cong C_c^{\infty}((R,\infty),C^{\infty}(Z)) \to C^{\infty}((R,\infty),C^{\infty}(Z)) \cong C^{\infty}(\{x > R\}),
$$
and by assumption on $A$ can find $a(x,\sigma,\tau) \in C_B^{\infty}(\R_+,M_{O,\textup{cl}}^{\mu;\vec{\ell}}(Z;\overline{\R}_+))$ and $G \in \Psi^{-\infty}_G(\R_+\times Z)$ such that
$$
[Au](x) = [\opm(a(x,\sigma,x))u](x) + [Gu](x), \quad x > R,
$$
for $u \in C_c^{\infty}((R,\infty),C^{\infty}(Z))$. For $x > R$ we then define
$$
\sym_e(A) = \sym(a)(z,x,\zeta,\sigma,\tau)
$$
for $(\zeta,\sigma,\tau) \neq 0$ as in \eqref{AeSym}. We have
$$
\sym_e(A)(z,x,\varrho^{\ell_1}\zeta,\varrho^{\ell_1}\sigma,\varrho^{\ell_3}\tau) = \varrho^{\mu}\sym_e(A)(z,x,\zeta,\sigma,\tau), \quad \varrho > 0,
$$
and
\begin{equation}\label{compatsymbol}
\sym_{\psi}(A)(z,x,\zeta,\xi) = \sym_e(A)(z,x,\zeta,x\xi,0).
\end{equation}
We then consider the tuple
$$
\sym(A):= (\sym_{\psi}(A),\sym_e(A)),
$$
where $\sym_e(A)$ only makes sense for $x > R \gg 1$ with $R$ depending on $A$, and where both $\sym_{\psi}(A)$ and $\sym_e(A)$ are coupled by \eqref{compatsymbol}, to be the full principal symbol of $A \in \Psi^{\mu;\vec{\ell}}_{\textup{cl}}(M)$. Observe that $\sym(A)$ determines $A$ modulo $\Psi^{\mu-1;\vec{\ell}}_{\textup{cl}}(M)$. It follows from the properties discussed below that
$$
0 \longrightarrow \Psi^{\mu-1;\vec{\ell}}_{\textup{cl}}(M) \longrightarrow \Psi^{\mu;\vec{\ell}}_{\textup{cl}}(M) \longrightarrow \Sigma^{\mu;\vec{\ell}} \longrightarrow 0
$$
is split-exact, where a right-inverse of the symbol map $\sym : \Psi^{\mu;\vec{\ell}}_{\textup{cl}}(M) \longrightarrow \Sigma^{\mu;\vec{\ell}}$ is given by patching quantization maps for $L^{\mu;\ell_1}_{\textup{cl}}(M)$ and $\Psi^{\mu;\vec{\ell}}_{\textup{cl}}(\R_+\times Z)$ (for large $x > R \gg 1$) with a partition of unity on $M$.

The following statements are immediate consequences from the definition and the properties of both $L^{\mu;\ell_1}_{\textup{(cl)}}(M)$ and $\Psi^{\mu;\vec{\ell}}_{\textup{(cl)}}(\R_+\times Z)$:

\begin{enumerate}
\item For any $\phi,\psi \in C_c^{\infty}(M)$ we have $\phi L^{\mu;\ell_1}_{\textup{(cl)}}(M) \psi \subset \Psi^{\mu;\vec{\ell}}_{\textup{(cl)}}(M)$.
\item For any $\phi,\psi \in C_c^{\infty}(M)$ with both $\textup{supp}(1-\phi)$ and $\textup{supp}(1-\psi)$ contained in $x > 1$ we have
$$
(1-\phi) \Psi^{\mu;\vec{\ell}}_{\textup{(cl)}}(\R_+\times Z)(1-\psi) \subset \Psi^{\mu;\vec{\ell}}_{\textup{(cl)}}(M).
$$
\item For any $\beta \in \R$ we have $x^{-\beta} \Psi^{\mu;\vec{\ell}}_{\textup{(cl)}}(M) x^{\beta} \subset \Psi^{\mu;\vec{\ell}}_{\textup{(cl)}}(M)$. If $A \in \Psi^{\mu;\vec{\ell}}_{\textup{cl}}(M)$, then  $\sym(x^{-\beta}Ax^{\beta}) = \sym(A)$.
\end{enumerate}

\noindent
The last property makes it possible to consider the calculus with weights $x^{\gamma} \Psi^{\mu;\vec{\ell}}_{\textup{(cl)}}(M)$ for $\gamma \in \R$. For $A \in x^{\gamma} \Psi^{\mu;\vec{\ell}}_{\textup{cl}}(M)$ we define $\sym(A):= \sym(x^{-\gamma}A)$.

\begin{theorem}\label{AdjointsM}
Let $A \in x^{\gamma} \Psi^{\mu;\vec{\ell}}(M)$. Then the formal adjoint $A^*$ with respect to the $L^2_b$-inner product on $M$ belongs to $x^{\gamma} \Psi^{\mu;\vec{\ell}}(M)$. If $A \in x^{\gamma} \Psi^{\mu;\vec{\ell}}_{\textup{cl}}(M)$ then also $A^* \in x^{\gamma} \Psi^{\mu;\vec{\ell}}_{\textup{cl}}(M)$, and we have $\sym(A^*) = \sym(A)^*$ (fiberwise adjoints).
\end{theorem}
\begin{proof}
It suffices to consider the case $\gamma = 0$. Write
$$
A = \phi A \psi + (1-\phi) A (1-\tilde{\psi}) + G,
$$
where $G \in \Psi^{-\infty}_G(M)$ and $\phi,\psi,\tilde{\psi} \in C_c^{\infty}(M)$ real-valued with both $\textup{supp}(1-\phi)$ and $\textup{supp}(1-\tilde{\psi})$ contained in $x > R$, where $R \gg 1$ is large, and $\tilde{\psi} \prec \phi \prec \psi$ (we write $g \prec f$ when $f \equiv 1$ in a neighborhood of the support of $g$).

Now $\phi A \psi \in L_{\textup{(cl)}}^{\mu;\ell_1}(M)$ with $[\phi A \psi]^* = \psi A^* \phi \in L_{\textup{(cl)}}^{\mu;\ell_1}(M)$, and as both $\phi$ and $\psi$ have compact support we get $[\phi A \psi]^* \in \Psi^{\mu;\vec{\ell}}_{\textup{(cl)}}(M)$, and in the classical case $\sym([\phi A \psi]^*) = \phi \sym(A)^*$. Similarly, for $R \gg 1$ large enough, $(1-\phi) A (1-\tilde{\psi}) \in \Psi^{\mu;\vec{\ell}}_{\textup{(cl)}}(\R_+\times Z)$, and by Theorems~\ref{AdjointsCalculus} and \ref{CompositionCalculus} we have $[(1-\phi) A (1-\tilde{\psi})]^* = (1 - \tilde{\psi}) A^* (1-\phi) \in \Psi^{\mu;\vec{\ell}}_{\textup{(cl)}}(\R_+\times Z)$. In the classical case, $\sym([(1-\phi) A (1-\tilde{\psi})]^*) = (1-\phi)\sym(A)^*$. In conclusion,
$$
A^* =  [\phi A \psi]^* + [(1-\phi) A (1-\tilde{\psi})]^* + G^* \in \Psi^{\mu;\vec{\ell}}_{\textup{(cl)}}(M)
$$
with $\sym(A^*) = \phi \sym(A)^* + (1-\phi) \sym(A)^* = \sym(A)^*$ in the classical case as claimed.
\end{proof}

\noindent
We note that Theorem~\ref{AdjointsM} implies that every $A \in x^{\gamma}\Psi^{\mu;\vec{\ell}}(M)$ gives a continuous operator
$$
A : \S'(M) \to \S'(M).
$$
Note that we consistently use the fixed density on $M$ to trivialize the density bundle, and hence identify functions and densities.

\begin{theorem}\label{CompositionM}
Let $A_j \in x^{\gamma_j} \Psi^{\mu_j;\vec{\ell}}_{\textup{(cl)}}(M)$, $j = 1,2$. Then the composition $A_1A_2$ belongs to $x^{\gamma_1+\gamma_2} \Psi^{\mu_1+\mu_2;\vec{\ell}}_{\textup{(cl)}}(M)$, and in the classical case $\sym(A_1A_2) = \sym(A_1)\sym(A_2)$ with componentwise multiplication. The residual class $\Psi^{-\infty}_G(M)$ is an ideal in the calculus.
\end{theorem}
\begin{proof}
It suffices to consider $\gamma_1=\gamma_2=0$.

For any $A \in \Psi^{\mu;\vec{\ell}}(M)$ there exists $K > 0$ sufficiently large such that
\begin{equation}\label{AMap}
A : x^{\alpha}H^{\frac{s}{\ell_1}}_b(M) \to x^{\alpha+K}H^{\frac{s-\mu}{\ell_1}}_b(M)
\end{equation}
is continuous for all $\alpha, s \in \R$. To see this write
$$
A = \phi A \psi + (1-\phi)A(1-\tilde{\psi}) + G
$$
with appropriate $\phi,\psi,\tilde{\psi} \in C_c^{\infty}(M)$ and $G \in \Psi^{-\infty}_G(M)$ as in the proof of Theorem~\ref{AdjointsM}. Now
$$
\phi A \psi : x^{\alpha}H^{\frac{s}{\ell_1}}_b(M) \to x^{\alpha+K}H^{\frac{s-\mu}{\ell_1}}_b(M)
$$
for all $s,\alpha \in \R$ and all $K \geq 0$ because $\phi$ and $\psi$ have compact support. Moreover, $G$ maps any weighted $b$-Sobolev space into any other weighted $b$-Sobolev space across the full range of weight and regularity parameters. Finally,
$$
(1-\phi)A(1-\tilde{\psi}) = \opm(a(x,\sigma,x)) + \tilde{G},
$$
where $\tilde{G} \in \Psi^{-\infty}_G(M)$, and we have
$$
x^{-K} a(x,\sigma,x) \in C^{\infty}_B(\R_+,M^{\mu;(\ell_1,\ell_1)}_O(Z))
$$
for $K \geq \frac{\mu_+}{\ell_3}$; note that the supports of both $1-\phi$ and $1-\tilde{\psi}$ are contained in $x > R$ with $R \gg 1$ sufficiently large. Consequently,
$$
x^{-K} (1-\phi)A(1-\tilde{\psi}) : x^{\alpha}H^{\frac{s}{\ell_1}}_b(M) \to x^{\alpha}H^{\frac{s-\mu}{\ell_1}}_b(M)
$$
for all $\alpha,s \in \R$ and $K \geq \frac{\mu_+}{\ell_3}$, and the claimed mapping property \eqref{AMap} for these values of $K$ follows.

The mapping properties \eqref{AMap} for $A$ combined with the characterizing mapping properties for $\Psi^{-\infty}_G(M)$ now show that $AG, GA \in \Psi^{-\infty}_G(M)$ for all $A \in \Psi^{\mu;\vec{\ell}}(M)$ and $G \in \Psi^{-\infty}_G(M)$.

We next consider the composition $A_1A_2$:
\begin{itemize}
\item Let $\phi,\psi \in C_c^{\infty}(M)$ be arbitrary with $\textup{supp}(\phi)\cap\textup{supp}(1-\psi) = \emptyset$, and let $\tilde{\phi} \in C_c^{\infty}(M)$ with $\phi \prec \tilde{\phi}$ and $\textup{supp}(\tilde{\phi})\cap\textup{supp}(1-\psi) = \emptyset$. Recall from the proof of Theorem~\ref{AdjointsM} that we write $\phi \prec \tilde{\phi}$ if $\tilde{\phi} \equiv 1$ in a neighborhood of $\textup{supp}(\phi)$. Then
$$
\phi A_1A_2 (1-\psi) = \phi A_1 [\tilde{\phi} A_2 (1-\psi)] + [\phi A_1 (1-\tilde{\phi})] A_2(1-\psi),
$$
where $[\tilde{\phi} A_2 (1-\psi)], [\phi A_1 (1-\tilde{\phi})] \in \Psi^{-\infty}_G(M)$. By the already proven ideal property of $\Psi^{-\infty}_G(M)$ we obtain that $\phi A_1A_2 (1-\psi) \in \Psi^{-\infty}_G(M)$. In the same way we prove that $(1-\psi)A_1A_2\phi \in \Psi^{-\infty}_G(M)$.
\item Let $\phi,\psi \in C_c^{\infty}(M)$ be arbitrary. Pick functions $\phi_1,\psi_1 \in C_c^{\infty}(M)$ with $\phi,\psi \prec \phi_1 \prec \psi_1$. Then
$$
\phi A_1 A_2 \psi = [\phi A_1 \psi_1][\phi_1 A_2 \psi] + [\phi A_1 (1-\phi_1)]A_2\psi.
$$
We have $[\phi A_1 \psi_1][\phi_1 A_2 \psi] \in L^{\mu_1+\mu_2;\ell_1}_{\textup{(cl)}}(M)$ with $\sym_{\psi}([\phi A_1 \psi_1][\phi_1 A_2 \psi]) = \phi\psi \sym_{\psi}(A_1)\sym_{\psi}(A_2)$ in the classical case, while $[\phi A_1 (1-\phi_1)]A_2\psi \in \Psi^{-\infty}_G(M)$ with compactly supported Schwartz kernel.
\item Pick $R \gg 1$ sufficiently large such that both $A_1$ and $A_2$, when restricted to $x > R$, arise as restrictions of operators in the calculus on $\R_+\times Z$ discussed in Section~\ref{PseudoCalculusExit}. Let then $\phi,\psi \in C_c^{\infty}(M)$ be arbitrary such that both $\textup{supp}(1-\phi)$ and $\textup{supp}(1-\psi)$ are contained in $x > R$. Let furthermore $\phi_1,\psi_1 \in C_c^{\infty}(M)$ with $(1-\phi),(1-\psi) \prec 1 - \phi_1 \prec 1 - \psi_1$, and both $\textup{supp}(1-\phi_1)$ and $\textup{supp}(1-\psi_1)$ still contained in $x > R$. Then
\begin{align*}
(1-\phi)A_1A_2(1-\psi) &= [(1-\phi)A_1(1-\psi_1)][(1-\phi_1)A_2(1-\psi)] \\
&\qquad + (1-\phi)A_1 [\phi_1A_2(1-\psi)].
\end{align*}
We have $[(1-\phi)A_1(1-\psi_1)][(1-\phi_1)A_2(1-\psi)] \in \Psi^{\mu_1+\mu_2;\vec{\ell}}_{\textup{(cl)}}(\R_+\times Z)$ by Theorem~\ref{CompositionCalculus}, and in the classical case
$$
\sym_e([(1-\phi)A_1(1-\psi_1)][(1-\phi_1)A_2(1-\psi)]) = (1-\phi)(1-\psi)\sym_e(A_1)\sym_e(A_2).
$$
We also have $(1-\phi)A_1 [\phi_1A_2(1-\psi)] \in \Psi^{-\infty}_G(M)$ with Schwartz kernel supported in $\{x > R\}\times \{x > R\}$, and so this operator is also in $\Psi^{-\infty}_G(\R_+\times Z)$.
\end{itemize}
The theorem is proved.
\end{proof}

\subsection{Ellipticity and parametrices}

As we did with the calculus on $\R_+\times Z$ in Section~\ref{PseudoCalculusExit}, we restrict the discussion of full ellipticity and the existence of parametrices in the calculus on $M$ to classical pseudodifferential operators.

\begin{definition}
\begin{enumerate}
\item The operator $A \in \Psi^{\mu;\vec{\ell}}_{\textup{cl}}(M)$ is fully elliptic if $\sym_{\psi}(A)$ is invertible on $T^*M\setminus 0$, and if, for sufficiently large $R \gg 1$, $\sym_e(A)(z,x,\zeta,\sigma,\tau)$ is invertible for all $(z,\zeta,\sigma,\tau) \in [T^*Z \times \R \times \overline{\R}_+] \setminus 0$ and all $x \geq R$, and the inverse
$$
|\sym^{-1}_e(A)(z,x,\zeta,\sigma,\tau)|
$$
is bounded for $x \geq R$ and $(z,\zeta,\sigma,\tau) \in T^*Z \times \R \times \overline{\R}_+$ with $|\zeta|_{g_Z(z)}^{2\ell_3} + |\sigma|^{2\ell_3} + |\tau|^{2\ell_1} = 1$, $z \in Z$, where $g_Z$ is any Riemannian metric on $Z$.
\item $A \in x^{\gamma}\Psi^{\mu;\vec{\ell}}_{\textup{cl}}(M)$ is fully elliptic if $x^{-\gamma}A \in \Psi^{\mu;\vec{\ell}}_{\textup{cl}}(M)$ is fully elliptic.
\end{enumerate}
\end{definition}

\begin{theorem}\label{Parametrix}
Let $A \in x^{\gamma}\Psi^{\mu;\vec{\ell}}_{\textup{cl}}(M)$ be fully elliptic. Then there exists a parametrix $P \in x^{-\gamma}\Psi^{-\mu;\vec{\ell}}_{\textup{cl}}(M)$ such that $AP-I,\, PA - I \in \Psi^{-\infty}_G(M)$.
\end{theorem}
\begin{proof}
We have existence of parametrices in the local calculus on ${\mathbb R}_+\times Z$, see Theorem~\ref{ParametrixModel}, and in the standard calculus for step-$\frac{1}{\ell_1}$ polyhomogeneous pseudodifferential operators, and in view of the algebra properties of the calculus on $M$ the standard patching argument is applicable to obtain a parametrix $P$ as asserted.
\end{proof}

\subsection{Sobolev spaces, elliptic regularity, and Fredholm operators}

\begin{proposition}\label{CompactBdd}
For $\mu \leq 0$ every $T \in \Psi^{\mu;\vec{\ell}}(M)$ is a continuous operator
$$
T : L^2_b(M) \to L^2_b(M).
$$
The operator $T$ is compact for $\mu < 0$.
\end{proposition}
\begin{proof}
We have already addressed boundedness of operators in $b$-Sobolev spaces in the proof of Theorem~\ref{CompositionM}. Because $\mu \leq 0$ we have
$$
T : x^{\alpha}H^{\frac{s}{\ell_1}}_b(M) \to x^{\alpha}H^{\frac{s-\mu}{\ell_1}}_b(M)
$$
for all $s,\alpha \in \R$. In particular, $T : L^2_b(M) \to L^2_b(M)$ is continuous.

To show the compactness of $T$ for $\mu < 0$ it suffices to prove that $(T^*T)^N : L^2_b(M) \to L^2_b(M)$ is compact for some $N \in \N$. By the spectral theorem for compact selfadjoint operators we then obtain that $|T| = [(T^*T)^N]^{\frac{1}{2N}} : L^2_b(M) \to L^2_b(M)$ is compact, and in view of the polar decomposition $T = U|T| : L^2_b(M) \to L^2_b(M)$ we see that $T$ is compact as well.

By Theorem~\ref{AdjointsM} and Theorem~\ref{CompositionM} we have $(T^*T)^N \in \Psi^{2\mu N;\vec{\ell}}(M)$ for every $N \in \N$, and so
$x(T^*T)^N \in \Psi^{2\mu N + \ell_3;\vec{\ell}}(M)$. For $N > -\frac{\ell_3}{2\mu}$ we have $s = -\frac{2\mu N + \ell_3}{\ell_1} > 0$, and we get a continuous operator $x(T^*T)^N : L^2_b(M) \to H^s_b(M)$. Consequently,
$$
(T^*T)^N : L^2_b(M) \to x^{-1}H^s_b(M)
$$
is continuous, and because the embedding $x^{-1}H^s_b(M) \hookrightarrow L^2_b(M)$ is compact, the operator $(T^*T)^N : L^2_b(M) \to L^2_b(M)$ is compact.
\end{proof}

\begin{proposition}\label{ReductionOrder}
For every $s \in \R$ there exists an operator $R \in \Psi^{s;\vec{\ell}}(M)$ that is invertible with inverse $R^{-1} \in \Psi^{-s;\vec{\ell}}(M)$.
\end{proposition}
\begin{proof}
Pick an asymptotically conic metric $g$ on $M$ that is of the form \eqref{ModelMetric} for large $x \gg 1$. On $[T^*M \times \R \times \overline{\R}_+] \setminus 0$ we consider the function
$$
p(y,\eta,\lambda,\tau) = \bigl([g(x\eta,x\eta)]^{\ell_3} + \lambda^{2\ell_3} + \tau^{2\ell_1}\bigr)^{\frac{s}{2\ell_1\ell_3}}.
$$
Then $p$ is invertible and satisfies
$$
p(y,\varrho^{\ell_1}\eta,\varrho^{\ell_1}\lambda,\varrho^{\ell_3}\tau) = \varrho^sp(y,\eta,\lambda,\tau), \quad \varrho > 0.
$$
Moreover, for large $x \gg 1$, we have $(y,\eta) \equiv (z,x,\zeta,\xi)$ and $g(x\eta,x\eta) = (x\xi)^2 + g_Z(\zeta,\zeta)$.

The symbol $p(y,\eta,\lambda,\tau)$ is a parameter-dependent elliptic principal symbol associated with the calculus of pseudodifferential operator families on $M$ that depend on the parameters $(\lambda,\tau)$. Quantizing yields an operator family $P(\lambda,\tau) \in L^{s;\vec{\ell}}_{\textup{cl}}(M;\R\times\overline{\R}_+)$ in that calculus with parameter-dependent principal symbol $p$. We now define
$$
[Q_0(\lambda)u](y) = [P(\lambda,x(y))u](y), \quad u \in C_c^{\infty}(M).
$$
Then $Q_0(\lambda)$ is a pseudodifferential operator family in the calculus $L^{s;\ell_1}_{\textup{cl}}(M;\R)$ of operators that depend on the parameter $\lambda \in \R$. The parameter-dependent principal symbol is $q(y,\eta,\lambda) = p(y,\eta,\lambda,0)$, which shows that $Q_0(\lambda)$ is parameter-dependent elliptic.

For large $x \gg 1$ we have
\begin{align*}
p(y,\eta,\lambda,\tau) &\equiv p(z,x,\zeta,\xi,\lambda,\tau) \\
&= \bigl([(x\xi)^2 + g_Z(\zeta,\zeta)]^{\ell_3} + \lambda^{2\ell_3} + \tau^{2\ell_1}\bigr)^{\frac{s}{2\ell_1\ell_3}} \\
&= \bigl([\sigma^2 + g_Z(\zeta,\zeta)]^{\ell_3} + \lambda^{2\ell_3} + \tau^{2\ell_1}\bigr)^{\frac{s}{2\ell_1\ell_3}}\Big|_{\sigma=x\xi}.
\end{align*}
The latter is a parameter-dependent elliptic extended principal symbol, and there exists an operator family $Q_1(\lambda) \in \Psi_{\textup{cl}}^{s;\vec{\ell}}(\R_+\times Z;\R)$ in the calculus on $\R_+\times Z$ that depends on the parameter $\lambda \in \R$ with
$$
\sym_e(Q_1) = \bigl([\sigma^2 + g_Z(\zeta,\zeta)]^{\ell_3} + \lambda^{2\ell_3} + \tau^{2\ell_1}\bigr)^{\frac{s}{2\ell_1\ell_3}}.
$$
The calculus on $\R_+\times Z$ with parameter $\lambda \in \R$ is built from symbols with additional symbolic dependence on $\lambda$, and an inspection of the arguments from Section~\ref{PseudoCalculusExit} shows that this calculus is well-defined. Let $\phi,\psi,\tilde{\psi} \in C_c^{\infty}(M)$ with $\tilde{\psi} \prec \phi \prec \psi$ and $\tilde{\psi} \equiv 1$ on $x \leq L$, where $L \gg 1$ is sufficiently large. Let then
$$
Q(\lambda) = \phi Q_0(\lambda)\psi + (1-\phi)Q_1(\lambda)(1-\tilde{\psi}) \in \Psi^{s;\vec{\ell}}_{\textup{cl}}(M;\R),
$$
where $\Psi^{s;\vec{\ell}}_{\textup{cl}}(M;\R)$ is the global calculus on $M$ that depends on an additional parameter $\lambda \in \R$. By construction $Q(\lambda)$ is fully elliptic with parameter, and thus there exists $T(\lambda) \in \Psi^{-s;\vec{\ell}}_{\textup{cl}}(M;\R)$ such that
$$
Q(\lambda)T(\lambda) - I,\,T(\lambda)Q(\lambda) - I \in \Psi^{-\infty}_G(M;\R).
$$
The parameter-dependent residual class $\Psi^{-\infty}_G(M;\R)$ consists of operator families $G(\lambda)$ such that
$$
G(\lambda) \in \S(\R,\L(x^{\alpha}H^s_b(M),x^{\alpha'}H^{s'}_b(M)))
$$
for all $s,s',\alpha,\alpha' \in \R$. This characterization via mapping properties shows that for any $G(\lambda) \in \Psi^{-\infty}_G(M;\R)$ there exists $G'(\lambda) \in \Psi^{-\infty}_G(M;\R)$ such that
$$
(I + G(\lambda))^{-1} = I + G'(\lambda)
$$
for large $|\lambda| \gg 0$. Consequently, for $|\lambda_0| \gg 0$ large enough, the operator
$$
R:= Q(\lambda_0) \in \Psi^{s;\vec{\ell}}_{\textup{cl}}(M)
$$
is invertible with inverse $R^{-1} = T(\lambda_0)(I + G'(\lambda_0)) \in \Psi^{-s;\vec{\ell}}_{\textup{cl}}(M)$ for some $G'(\lambda) \in \Psi^{-\infty}_G(M;\R)$.
\end{proof}

\begin{definition}
Pick any invertible operator $R \in \Psi^{s;\vec{\ell}}(M)$ with $R^{-1} \in \Psi^{-s;\vec{\ell}}(M)$, and define
$$
{\mathscr H}^{s;\vec{\ell}}(M) = \{u \in \S'(M);\; Ru \in L^2_b(M)\}
$$
with inner product $\langle u,v \rangle_{{\mathscr H}^{s;\vec{\ell}}}:= \langle Ru,Rv \rangle_{L^2_b}$. This is a Hilbert space, and while the inner product depends on $R$, different choices of operators $R$ yield the same Sobolev space with equivalent norms. Moreover, $\S(M) \subset {\mathscr H}^{s;\vec{\ell}}(M)$ is dense for every $s \in \R$.

We also consider the weighted spaces $x^{\alpha}{\mathscr H}^{s;\vec{\ell}}(M)$ for $\alpha,s \in \R$.
\end{definition}

\begin{theorem}\label{SobolevSpaceProperties}
\begin{enumerate}[(a)]
\item Every $A \in x^{\gamma}\Psi^{\mu;\vec{\ell}}(M)$ is continuous in
$$
A : x^{\alpha}{\mathscr H}^{s;\vec{\ell}}(M) \to x^{\alpha+\gamma}{\mathscr H}^{s-\mu;\vec{\ell}}(M) 
$$
for all $s,\alpha \in \R$.
\item The embedding $x^{\alpha}{\mathscr H}^{s;\vec{\ell}}(M) \hookrightarrow x^{\alpha'}{\mathscr H}^{s';\vec{\ell}}(M)$ is continuous for $s \geq s'$ and $\alpha \leq \alpha'$. Furthermore, the embedding is compact for $s > s'$ and $\alpha \leq \alpha'$.
\item The $L^2_b$-inner product extends to a continuous and perfect sesquilinear pairing
$$
\langle \cdot,\cdot \rangle : x^{\alpha}{\mathscr H}^{s;\vec{\ell}}(M) \times x^{-\alpha}{\mathscr H}^{-s;\vec{\ell}}(M) \to \C
$$
for all $s,\alpha \in \R$.
\item We have
$$
H^{\frac{s}{\ell_1}}_{\textup{comp}}(M) \subset {\mathscr H}^{s;\vec{\ell}}(M)  \subset H^{\frac{s}{\ell_1}}_{\textup{loc}}(M) 
$$
for all $s \in \R$, and
$$
\S(M) = \bigcap\limits_{s \in \R}{\mathscr H}^{s;\vec{\ell}}(M) \textup{ and } \S'(M) = \bigcup\limits_{s \in \R}{\mathscr H}^{s;\vec{\ell}}(M).
$$
\end{enumerate}
\end{theorem}
\begin{proof}
Parts (a)--(c) follow from Propositions~\ref{CompactBdd} and \ref{ReductionOrder} and the algebra properties of the calculus. For (d) note that for $s \geq 0$ we have
$$
x^{-\frac{s}{\ell_3}}H^{\frac{s}{\ell_1}}_b(M) \hookrightarrow {\mathscr H}^{s;\vec{\ell}}(M) \hookrightarrow x^{-j}H^{\frac{s-j\ell_3}{\ell_1}}_b(M)
$$
for $0 \leq j \leq \frac{s}{\ell_3}$, $j \in \N_0$, and for $s < 0$ we have
$$
x^{j}H^{\frac{s+j\ell_3}{\ell_1}}_b(M) \hookrightarrow {\mathscr H}^{s;\vec{\ell}}(M) \hookrightarrow x^{-\frac{s}{\ell_3}}H^{\frac{s}{\ell_1}}_b(M)
$$
for $0 \leq j \leq -\frac{s}{\ell_3}$, $j \in \N_0$.
\end{proof}

\begin{theorem}
Let $A \in x^{\gamma}\Psi_{\textup{cl}}^{\mu;\vec{\ell}}(M)$ be fully elliptic.
\begin{enumerate}[(a)]
\item For all $s,\alpha \in \R$ the operator
$$
A : x^{\alpha}{\mathscr H}^{s;\vec{\ell}}(M) \to x^{\alpha+\gamma}{\mathscr H}^{s-\mu;\vec{\ell}}(M) 
$$
is Fredholm.
\item Let $u \in \S'(M)$ such that $Au \in x^{\alpha}{\mathscr H}^{s;\vec{\ell}}(M)$ for some $\alpha,s \in \R$. Then $u \in x^{\alpha-\gamma}{\mathscr H}^{s+\mu;\vec{\ell}}(M)$.
\item The decompositions
\begin{align*}
x^{\alpha}{\mathscr H}^{s;\vec{\ell}}(M) &= \ker(A) \oplus \ran[A^* : x^{\alpha-\gamma}{\mathscr H}^{s+\mu;\vec{\ell}}(M) \to x^{\alpha}{\mathscr H}^{s;\vec{\ell}}(M)], \\
x^{\alpha+\gamma}{\mathscr H}^{s-\mu;\vec{\ell}}(M) &= \ker(A^*) \oplus \ran[A : x^{\alpha}{\mathscr H}^{s;\vec{\ell}}(M) \to x^{\alpha+\gamma}{\mathscr H}^{s-\mu;\vec{\ell}}(M)]
\end{align*}
are topologically direct sums, and $\ker(A),\,\ker(A^*) \subset \S(M)$ are independent of $s,\alpha \in \R$.
\item Let $\pi = \pi^2 =\pi^* \in \Psi^{-\infty}_G(M)$ be the $L^2_b$-orthogonal projection onto $\ker(A)$, and $\tilde{\pi} = \tilde{\pi}^2 = \tilde{\pi}^* \in \Psi^{-\infty}_G(M)$ be the $L^2_b$-orthogonal projection onto $\ker(A^*)$. The Moore-Penrose inverse (pseudoinverse) of $A$
$$
P = (\pi + A^*A)^{-1}A^* = A^*(\tilde{\pi} + AA^*)^{-1} \in x^{-\gamma}\Psi_{\textup{cl}}^{-\mu;\vec{\ell}}(M),
$$
and $PA = I - \pi$ and $AP = I - \tilde{\pi}$. The projection operators $AP,\, PA \in \Psi_{\textup{cl}}^{0;\vec{\ell}}(M)$ are (formally) selfadjoint, and their continuous extensions to the weighted Sobo\-lev spaces are the projections onto the ranges of $A$ and $A^*$ in the direct decompositions of the spaces given above.
\end{enumerate}
\end{theorem}
\begin{proof}
Parts (a) and (b) follow from Theorem~\ref{Parametrix} and Theorem~\ref{SobolevSpaceProperties}. To prove (d) note that both
$$
(\pi + A^*A),\, (\tilde{\pi} + AA^*) \in x^{2\gamma}\Psi_{\textup{cl}}^{2\mu;\vec{\ell}}(M)
$$
are fully elliptic and formally selfadjoint with
$$
\ker(\pi + A^*A) = \{0\} = \ker(\tilde{\pi} + AA^*).
$$
To see the latter observe that both kernels are contained in $\S(M)$ by elliptic regularity (b), and $(\pi + A^*A) : \S(M) \to \S(M)$ is injective since
$$
\langle (\pi + A^*A)u,u \rangle_{L^2_b} = \|\pi u\|_{L^2_b}^2 + \|Au\|_{L^2_b}^2, \quad u \in \S(M),
$$
and likewise for $(\tilde{\pi} + AA^*)$. The Fredholm property implies that the range of each of
$$
(\pi + A^*A),\, (\tilde{\pi} + AA^*) : x^{\alpha}{\mathscr H}^{s;\vec{\ell}}(M) \to x^{\alpha+2\gamma}{\mathscr H}^{s-2\mu;\vec{\ell}}(M) 
$$
is closed for any $\alpha,s \in \R$, and by duality of the Sobolev space scale with respect to the $L^2_b$-inner product and injectivity of the formal adjoints we then obtain that these operators are surjective, hence invertible.

Let $Q \in x^{-2\gamma}\Psi_{\textup{cl}}^{-2\mu;\vec{\ell}}(M)$ be a parametrix for $(\pi + A^*A)$ such that
$$
(\pi + A^*A)Q = I + G_R, \quad Q(\pi + A^*A) = I + G_L, \quad G_L,\,G_R \in \Psi^{-\infty}_G(M).
$$
Then
$$
(\pi + A^*A)^{-1} = Q - QG_R + G_L(\pi + A^*A)^{-1}G_R.
$$
Now $G_L(\pi + A^*A)^{-1}G_R \in \Psi^{-\infty}_G(M)$ as this operator satisfies the defining mapping properties of the residual class, and so $(\pi + A^*A)^{-1} \in x^{-2\gamma}\Psi_{\textup{cl}}^{-2\mu;\vec{\ell}}(M)$, and likewise $(\tilde{\pi} + AA^*)^{-1} \in x^{-2\gamma}\Psi_{\textup{cl}}^{-2\mu;\vec{\ell}}(M)$. Thus
$$
P = (\pi + A^*A)^{-1}A^* = A^*(\tilde{\pi} + AA^*)^{-1} \in x^{-\gamma}\Psi_{\textup{cl}}^{-\mu;\vec{\ell}}(M)
$$
as asserted. Both (d) and (c) follow.
\end{proof}


\end{document}